\documentclass[11pt]{amsart}
\usepackage{amsmath,amscd,amssymb}
\usepackage{color}

\allowdisplaybreaks
\numberwithin{equation}{section}

\newtheorem{theorem}{Theorem}[section]
\newtheorem{lemma}[theorem]{Lemma}
\newtheorem{proposition}[theorem]{Proposition}
\newtheorem{corollary}[theorem]{Corollary}

\theoremstyle{remark}
\newtheorem{remark}[theorem]{Remark}

\newtheorem{example}[theorem]{Example}

\numberwithin{equation}{section}
\numberwithin{theorem}{section}
\def\e{\varepsilon}
\def\CC{{\mathbb C}}
\def\RR{{\mathbb R}}
\def\NN{{\mathbb N}}
\def\codim{{\rm codim}\,}
\def\dist{{\rm dist}\,}
\def\Re{{\rm Re}\,}

\def\la{\lambda}
\def\al{\alpha}
 
\def\codim{{\rm codim}\,}
\def\ph{\varphi}

\def\NN{{\mathbb N}}
\def\CC{{\mathbb C}}
\def\RR{{\mathbb R}}
\def\Re{{\rm Re}\,}

\def\cM{\mathcal{M}}

\begin{document}
\title[On growth and global instability]{On growth and instability for semilinear evolution equations: an abstract approach}

\author{Vladimir M\"uller}
\address{V.M, Institute of Mathematics,
Czech Academy of Sciences,
ul. \v Zitna 25, Prague,
 Czech Republic}
\email{muller@math.cas.cz}

\author{Roland Schnaubelt}
\address{R.S., Karlsruhe Institute of Technology, Department of Mathematics, 76128 Karlsruhe, Germany}
\email{schnaubelt@kit.edu}

\author{Yuri Tomilov}
\address{Y.T., Institute of Mathematics, Polish Academy of Sciences,
\' Sniadeckich str.8, 00-656 Warsaw, Poland}
\email{ytomilov@impan.pl}

\address{
Faculty of Mathematics and Computer Science\\
Nicolas Copernicus University\\
ul. Chopina 12/18\\
87-100 Toru\'n, Poland
}


\date{\today}

\thanks{The work of the first author was  partially supported 
by grant  No.20-22230L and RVO:67985840. The second author was funded by the Deutsche Forschungsgemeinschaft 
(DFG, German Research Foundation) -- Project-ID 258734477 -- SFB 1173. The third author was partially
supported by the NCN grant 2017/27/B/ST1/00078.}


\begin{abstract}
We propose a new approach to the study of (nonlinear)
growth and instability for semilinear abstract evolution equations with compact nonlinearities. We show,
in particular, that compact nonlinear perturbations of linear evolution
equations can be treated as linear ones as far as the growth of their 
solutions is concerned. We obtain exponential lower bounds of solutions for initial values
from a dense set in resolvent or spectral terms. The abstract results are applied, in particular, to the
study of energy growth for semilinear backward damped wave equations.
\end{abstract}

\maketitle
\section{Introduction}\label{intro}
The paper is devoted to the study of instability of solutions to semilinear evolution equations
\begin{equation}\label{solaeq}
x'(t)=A x(t)+K(t, x(t)), \quad t \ge 0, \qquad x(0)=x_0 \in X,
\end{equation}
on a Banach space $X$, where $A$ generates a $C_0$-semigroup on $X$
and $K$  is a nonlinear map on $X$ subject to appropriate conditions
ensuring the existence of global mild solutions to \eqref{solaeq}.
While the problem of finding instability conditions for \eqref{solaeq}
in terms of $A$ and $K$ is of fundamental importance, very few
results in this direction were obtained so far even when $K$ is stationary, i.e., $K(t,x)=K(x)$
for $t\ge 0.$ 

One of the basic and commonly used instability criteria is due to Shatah and Strauss, \cite{Shatah1}.
It says that the zero solution of the system governed by \eqref{solaeq} is  unstable %near the origin
if the spectral bound of $A$ is greater than zero and $K$ is small enough in a metric sense 
(that is,  $K(x)={\rm O}(\|x\|^{1+\delta})$ with $\delta>0$ as  $\|x\| \to 0$) and continuous.
The notion of smallness can be refined by replacing a polynomial bound for $K$ with a bound
of more general type, but the scheme of proof remains the same all the time.
For bounded $A$ such criteria go back to \cite{dalec}.
This linearized instability principle appeared to be very useful
in a great number of applications.
Other versions of such results are investigated in \cite{Grill0}, \cite{Grill1}, and \cite{Grill2}.

Recently, the problem of keeping control over asymptotic behavior of
trajectories of linear infinite-dimensional systems under ``small'' nonlinear
perturbations was revived in \cite{Gallay}, \cite{Garab}, and \cite{Sola_Rod}.
There the opposite problem of stabilization by the nonlinearity was emphasized.
Note that these papers treated the situation of discrete time,
where a number of difficulties (e.g.\ failure of the spectral mapping theorem)
is missing.

In this paper, we address the situation when $K$ is a compact nonlinear map, i.e., 
$K$ is small in a topological sense, but otherwise it can be large metrically.  This type 
of perturbations allows for global instability results, instead of just local ones. 
Recall that the map $K: X \to X$ is called 
\emph{compact} if it is continuous (this assumption varies) and maps bounded sets to precompact sets.
Assuming that $K$ is a compact $C^1$-map, the paper \cite{Sola} considered the case that either the (Browder)
essential spectrum of $A$ intersects the open right half-plane $\CC_+$ or that there are infinitely many 
eigenvalues of $A$ in $\e +\CC_+$ for some $\e >0$. Then for a residual set of initial values $x_0$ 
the mild solutions $\{x(t, x_0): 0\le t < T\}$ of \eqref{solaeq} are unbounded on their maximal existence
interval $[0,T)$, where $0< T \le \infty.$ The proof relied on a discrete version of this result:
If $A$ is a bounded linear operator on $X$ with the essential spectral radius $r_{e}(A)>1$ and $T=A+K,$
then the trajectory $\bigcup_{n \in \mathbb N} T^n(B)$ of the unit ball $B \subset X$ is unbounded in $X$.

We substantially improve these results and, for the continuous time version \eqref{solaeq},  
replace the spectral terms used in \cite{Sola} by much weaker assumptions on resolvent bounds.
While the results in \cite{Sola} yield merely unboundedness (and then instability), we derive an optimal
exponential lower bound. To the best of our knowledge,  such results were absent in the literature.
In fact, our approach goes much further and allows one to treat $K$ depending on time.
Moreover, the obtained results cannot be improved even for stationary linear compact perturbations,
see Remark~\ref{rem:opt}.
To this aim we use techniques for the study of orbits of linear operators from \cite{Muller_Tom}
that, with certain modifications, work in the context of nonlinear maps as well.
This approach is used for the treatment of nonlinear equations for the first time, and
we hope it will be useful in a number of instances.

The following statement is an important partial case of our main results. See Theorem~\ref{main} combined with
Theorem~\ref{proplower}.
(A linear version of this result is given in Corollary \ref{linear}.)

\begin{theorem}\label{main0}
Let $A$ generate a $C_0$-semigroup $(T(t))_{t\ge 0}$  on a Banach space $X,$
with $\alpha:=\limsup_{t \to 0}\|T(t)\|,$
and let $K:[0,\infty) \times X \to X$
be jointly continuous and  map bounded sets in precompact ones.
Let $B\subset X$ be a bounded set. Assume that
for every $x_0\in B$ there exist a global mild  solution $x(t,x_0)$, $t\in[0,\infty),$ of \eqref{solaeq}
such that the set $S(t_0,B)=\{x(t, x_0):\, 0\le t\le t_0, \,\, x_0 \in B\}$
is bounded for each $t_0>0$.
Let $a:[0,\infty)\to [0,\infty)$ be non-increasing  with
$\lim_{t\to\infty}a(t)=0$, and let $r>0$, $y\in B$ and $t_0=t_0(r)\ge 0$  be such that $a(t_0)<\frac{r}{2\alpha}$
and $\overline{B}(y,r)\subset B$.
\begin{itemize}
\item [(i)] Assume that the resolvent $(\omega+ib-A)^{-1}$ is well-defined and unbounded for $b\in\RR$
and some  $ \omega \in\RR$.
Then there exist $x_0 \in \overline{B}(y,r)$ and a mild solution $x(t,x_0)$ of \eqref{solaeq}  such that
$$
\|x(t,x_0)\|\ge a(t) e^{\omega t}, \qquad t\ge t_0.
$$
\item [(ii)] In any case,
there exist $x_0 \in \overline{B}(y,r)$ and a mild solution $x(t,x_0)$ of
\eqref{solaeq}  such that
$$
\|x(t,x_0)\|\ge a(t) e^{s_e(A) t}, \qquad t\ge t_0,
$$
\end{itemize}
where $s_e(A)$ is the supremum of the real parts of $\lambda$ from the essential
spectrum of $A$.
\end{theorem}
In the above theorem one can replace the assumption on existence of
global mild solutions $x(t, x_0)$ with bounded sets $S(t_0,B)$ by the (linear) growth assumption
 $\|K(t,x)\|\le (1+c(t))\|x\|$ for some $c\in L^1_{\mathrm{loc}}([0,\infty))$ and all $t\ge0$
and $x\in X$. See Proposition~\ref{prop:growth} and Corollary~\ref{thm:cp2}.
In the case of backward damped wave equations considered in Section~\ref{sec:exa}, the condition of linear growth
of $K$ will be generalised by requiring that $K$ satisfies a sign condition and possesses some extra regularity.

To the best of our knowledge, the result is genuinely new in three respects:
\begin{itemize}
\item [(a)] it is the first lower bound for global growth of solutions to a general class of  nonlinear evolution equations in the
    literature (and in fact, it is new even in the setting of linear equations);
\item [(b)] the result is formulated in explicit (a priori) spectral terms;
\item [(c)] the results is sharp, even in a linear context (see Remark~\ref{rem:opt}).
\end{itemize}

To provide a necessary insight and because of independent interest,
we first develop a similar theory for the discrete counterpart of \eqref{solaeq}
\begin{equation}\label{solaeq1}
x_{n+1}= A x_n + K_n (x_n), \qquad x_0=x\in X, \qquad n \ge 0,
\end{equation}
where now $A$ is a bounded linear operator on a Banach space $X,$ and $K_n$ are compact maps on $X.$
In particular, we obtain an analogue of Theorem \ref{main0}  for  \eqref{solaeq1} only assuming
compactness of each map $K_n$, where  the lower bounds depend on the
essential spectral radius of the operator coefficient. 
\begin{theorem}\label{main10}
Let $A$ be a bounded linear operator on a Banach space $X$, $(K_n)_{n=1}^\infty$ be a sequence of
compact maps on $X$, and $(x_n(x_0))_{n=1}^\infty$ be given by \eqref{solaeq1}.
Take a non-increasing sequence $(a_n)_{n=1}^{\infty}\subset \mathbb R_+$ satisfying
$\lim_{n\to\infty}a_n=0$. Fix $y\in X$, $r>0$, and $n_0\in \mathbb N$ with
$a_{n_0}<\frac{r}{2}$. Then there exists  $x_0\in B(y,r)$ such that
$$
\|x_n(x_0)\|\ge a_n r_{e}(A)^n,\qquad n\ge n_0,
$$
where $r_{e}(A)$ is the essential spectral radius of $A.$
\end{theorem} 
This result is proved in Theorem~\ref{theorem1}. If $r_e(A)>1,$ then the orbits $x_n(x_0)$ grow exponentially
for a dense set of initial values $x_0.$ Thus, if the nonlinear part $K$ is small in a topological sense rather
than in the sense of norm,
Theorem \ref{main10} provides a global generalization of the classical (discrete) principle of  linearized instability,
as discussed e.g.\ in \cite[Theorem 5.1.5]{Henry} or \cite[Sections 1.2 and 4]{Gallay}.
Note that the proofs of a number of statements on instability for continuous time are reduced to similar considerations
in the discrete setting (as e.g.\ in \cite{Henry}). Also our discrete results are essentially optimal,
cf.\ \cite[Sections V.37 and V.39]{Muller} for the case $K=0$.

Using measures of non-compactness, we further obtain similar results producing residual sets of solutions
to \eqref{solaeq} and \eqref{solaeq1} with exponentially growing orbits. See Theorems~\ref{subseq} and \ref{subeq}, respectively.

To explain the general effects, we note that as remarked already in \cite{Sola},
 roughly speaking, the linear semigroup $(T(t))_{t \ge 0}$ is expanding along
infinitely many independent directions and the non-linear perturbation $K$ (being
relatively compact) is not able to compensate this expansion, except perhaps in a
finite number of directions. Our results say that up to a small multiplicative correction the expansion
takes place as if the nonlinear part in \eqref{solaeq} is absent.

The next toy example illustrates the specifics of the infinite-dimensional setting very well.
Let $X$ be a separable Hilbert space.
In the Banach space $B(X)$ of bounded linear operators on $X,$ consider the difference equation
\begin{equation*}
Y_{n+1}=A Y_n + K Y_n^2, \qquad Y_0 \in B(X), \qquad n \in \mathbb N.
\end{equation*}
If $X$ is finite-dimensional, then even if ${\rm dim} \, X = 1$ the asymptotic behavior of $\{Y_n\}$ could be
extremely complicated  (see e.g.\ \cite{Block}), and any bounds for $\|Y_n\|$ can hardly put under control for
\emph{individual} $Y_0.$ However, if ${\rm dim} \, X = \infty,$ and $K \in B(X)$ is compact, then the quadratic
part $Y\mapsto K Y^2$, being compact in $B(X),$ becomes ``small'' with respect to the linear part $Y \mapsto AY$.
Hence, by our Theorem \ref{theorem1} below,  the (exponential) growth of trajectories $\{Y_n\}$  for a residual set
of initial values is determined by the essential spectral radius of the linear part $L_A: Y\mapsto AY$
(i.e., by the essential spectral radius of $A$ if one notes that 
$\sigma_{e}(L_A)=\sigma(A)$ by e.g.\ \cite[Theorem 3.1]{Fialkow}).

 Aiming at generalizations of local instability results as in e.g.\
\cite{dalec}, \cite{Henry}, \cite{Shatah1} and \cite{Gallay}, it is natural to consider also nonlinear
perturbations $K+G$ with a compact operator $K$ and
 a metrically small map $G:B(0,r)\to X$ satisfying $G(x)=O(\|x\|^{1+\delta})$ as $\|x\| \to 0.$
 However, at least in the time discrete case given by  \eqref{solaeq1} we show
 in Section~\ref{local_discr} that  a local instability result analogous to Theorem~\ref{main10}
 cannot hold for such perturbations $K+G$.

As an illustration of our abstract results, we  apply them to certain ``excited'' and backward damped wave equations
in Examples~\ref{sola}, \ref{cox}, \ref{sjostrand}, and \ref{dyatlov}. Here we allow for nonlinear forcing terms $f(t,x,u)$ in
which the scalar function $f$ may grow superlinearly (if it has the  right sign) and is only continuous
in the last argument.
In general, the study of damped wave equations is an extremely vast and challenging area of research with
many open problems stemming from mathematical physics and control theory.
One may consult, e.g., the books \cite[Chapter 6]{LeR} and \cite[Chapters 10, 11]{LeR1} the survey \cite{Chill}, and the papers
  \cite{Anan_L}, \cite{Asch}, \cite{BurqC},  \cite{Chill1}, \cite{Dyatlov}, \cite{Jin}, \cite{JL},
 \cite{JL1}, \cite{Nonnen}, \cite{Riviere}, \cite{SchenckP}, \cite{Schenck}, \cite{Sjostrand}
for some of its developments in the linear setting, relevant for the nonlinear studies in the present paper.
However, we are not aware of any results similar to Theorem \ref{main0} in the context of nonlinear damped wave equations.

We believe there are many other frameworks, where our instability criteria could be useful, e.g., in the context of
reaction-diffusion systems as in \cite{Sola}. However, they would require a separate treatment.

In Sections~\ref{prelim}, \ref{sec:warmup}, and \ref{sec:app} we provide tools for our
analysis and discuss the background. The main results are proved in Section~\ref{discrete} and \ref{continuous}
for discrete and continuous time, respectively.
Section~\ref{local_discr} contains counterexamples to local results in discrete time.
The necessary information on nonlinear evolution equations
is collected in Section~\ref{sec:wp} and then used in Section~\ref{sec:exa} for our examples.

Finally, we fix some notation used throughout the paper.
All of the Banach (and Hilbert) spaces considered in this paper will be complex.
To avoid trivialities, we will always assume that these spaces are infinite-dimensional.
For a densely defined closed operator $A$ on a Banach space we denote
by  $\rho(A)$ its resolvent set, by $\sigma(A)$ its spectrum, % by $\sigma_e(A)$ its essential spectrum,
and by $\sigma_p(A)$ its point spectrum. We let $D(A)$ stand for the domain of $A,$ ${\rm Ker}\, (A)$
for the kernel of $A,$  ${\rm Im}(A)$ for its range, and $R(\lambda,A)=(\lambda-A)^{-1}$ for the resolvent of $A.$
The space of bounded linear operators on a Banach space $X$ will be denoted by $B(X)$ and that of compact
linear operators by $K(X)$.
For a subspace $M$ of a Banach space, ${\rm dim}\, M$ is the dimension of $M,$ and ${\rm codim}\, M$
its codimension.  For a subset $S$ of a topological space,  $\partial S$ designates its boundary,
and ${\rm card}\, S$  the cardinality of an arbitrary set $S.$

\section{Preliminaries: a toolkit for getting (nonlinear) instability}\label{prelim}

In this section, we review several tools and techniques used for deriving instability.
Some of them appear in the context of nonlinear evolution equations for the first time.

\subsection{Fine spectral theory}
We start with a short reminder of fine spectral theory. Recall that for a closed, densely defined linear 
operator $A$ on a Banach space $X$ 
its essential spectrum $\sigma_{e}(A)$ is defined as
\begin{equation*}
\sigma_{e}(A)=\{\lambda \in \mathbb C: \lambda -A \,\, \text{is not Fredholm} \}.
\end{equation*}
Clearly, $\sigma_e(A)\subset \sigma (A).$ Moreover, by \cite[Theorem 7.25]{Schechter}, $\sigma_e(A)$ is closed 
(but can be empty). There are many other ``essential spectra'' in the literature,
e.g.\ Browder's essential spectrum used in \cite{Sola}. A crucial property of 
$\sigma_{e}(A)$ is that it is invariant is invariant under relatively compact perturbations if $\rho(A) \neq \emptyset,$
see e.g.\ \cite[Theorem 11.2.6]{Davies}. Moreover, by \cite[Theorem 11.2.2]{Davies}, if there exists $\mu \in \rho(A),$ then
\begin{equation}\label{smp_resolv}
\sigma_e(R(\mu,A))\setminus \{0\}=\{(\mu-\lambda)^{-1}: \,\, \lambda \in \sigma_e(A)\}.
\end{equation}
The property is a consequence of a more general spectral mapping theorem for essential spectrum (\cite{Gramsch}), and it allows one to reduce many statements on essential spectrum for unbounded operators to their 
counterparts for bounded ones.

If $A$ is bounded, then $\sigma_{e}(A)$ is
a non-empty compact subset of $\mathbb C.$ In this case, if $r_{e}(A)$ denotes the essential spectral radius
of $A,$ then one  has  $r_{e}(A^n)=r_{e}(A)^n$ for all $n\in\NN$,
as a consequence of the spectral mapping theorem for the essential spectrum.
Note that
$\sigma(A)\setminus \{\lambda: |\lambda|\le r_e(A)\}$ consists of at most countably many isolated eigenvalues (of finite multiplicity), see e.g.\ \cite[Theorem III.19.4]{Muller}.

Some generalizations of Fredholm operators will also play a role. Recall that a closed, densely defined operator 
$A$ on $X$ is called upper-Fredholm if ${\rm dim} \, {\rm Ker}\, (A) < \infty$ and ${\rm Im}\, (A)$ is closed. 
If  $\sigma_e(A)$ is large and $\rho(A)\neq \emptyset,$ then the set of $\lambda \in \mathbb C$ such that  $\lambda-A$
is not upper Fredholm is large as well, since it contains the topological boundary $\partial \sigma_e(A).$
More precisely, if $\lambda \in \partial \sigma_{e}(A)$, then for every $\e >0$ and every closed
subspace $M \subset X$ of finite codimension there exists
a unit vector $u \in M \cap D(A)$ such that $\|(A-\lambda)u\|<\e.$  See Lemma~\ref{codim}.

The following useful proposition can be found in e.g.\ \cite[Theorem 9.43]{Schechter}. As several arguments below
hold ``up to compact perturbations'', the proposition allows one to deal with point spectrum rather than generic 
essential spectrum, and that is technically more convenient.
\begin{proposition}\label{Sch}
Let $A \in B(X).$ The operator $A$ has closed range and finite-dimensional kernel (i.e., $A$ is upper Fredholm)
 if and only if ${\rm dim}\, {\rm Ker}\, (A+K)<\infty$ for all compact operators $K$ on $X.$
\end{proposition}
While there is a version of Proposition \ref{Sch} for unbounded $A,$ the version above will be sufficient
for our purposes.

\subsection{Measures of non-compactness}
As far the essential spectrum is involved, measures of non-compactness
naturally come into play, although their role in our studies will rather be supplementary
in contrast to \cite{Sola}.
For $A\in B(X)$ let
$$
\|A\|_\mu=\inf\bigl\{\|A \upharpoonright_M\|: M\subset X \,\, \text{is a closed subspace},\, \,\codim M<\infty\bigr\}.
$$
 The mapping $A \mapsto \| A \|_\mu$ is called $\mu$-measure of non-compactness on $B(X).$ 
 It defines a semi-norm on $B(X),$ where $\|A\|_\mu=0$  if and only if $A$ is compact.
Moreover, we have $\|AB\|_\mu\le\|A\|_\mu\cdot\|B\|_\mu$ for all $A,B\in B(X).$

For $A \in B(X)$, we introduce the essential norm of $A$ by
$$
\|A\|_{e}:= \inf\{\|A-K\|:K \in K(X)\};
$$
 i.e., $\|A\|_{e}$  is the norm of the image of $A$ in the Calkin algebra $B(X)/K(X)$
under the corresponding quotient map.  
If $X$ is a Hilbert space, then according to \cite{Ylinen} (see also \cite{Webb}) we have
\begin{equation*}
\|A\|_\mu = \|A\|_{e}.
\end{equation*}
In general, $\|A\|_{\mu}$ and $\|A\|_{e}$ are equivalent norms on $B(X)/K(X)$ if and only if
$X$ has a so-called compact approximation property, see \cite{Astala}. While substantial classes of Banach
spaces possess this property, there are reflexive Banach spaces failing to satisfy it.
However, by the well-known Nussbaum formula
\begin{equation*}
r_{e}(A)=\lim_{n \to \infty}\|A^n\|_{\mu}^{1/n}=\lim_{n \to \infty}\|A^n\|_{e}^{1/n},
\end{equation*}
valid for all definitions of the essential spectrum, the quantities $\|\cdot\|_{\mu}$ and $\|\cdot\|_{e}$
are asymptotically equivalent in a sense. Moreover, the limits above can be replaced with the infimums. 
Hence, $\|A\|_{\mu}^n \ge r_{e}(A)^n$ and $\|A\|_{e}^n\ge r_{e}(A)^n$ for all $n \in \mathbb N.$

\subsection{Spectral theory for operator semigroups and resolvent bounds}

The spectral theory for $C_0$-semigroups is rather involved
due to the unboundedness of their generators, and one has to invoke
the size of resolvents to partially remedy the situation.

First recall that if  $(T(t))_{t \ge 0}$ a $C_0$-semigroup on a Banach space $X,$
with generator $A,$ then
\begin{equation}\label{inclus}
e^{t\sigma_{e}(A)}\subset \sigma_{e}(T(t))\qquad \text{for all}\quad t \ge 0,
\end{equation}
see e.g.\ \cite{Nagy} for even finer versions of the above inclusion.
This inclusion is strict in general. One can replace  in \eqref{inclus} 
the essential spectrum by the spectrum, where again the inclusion can be  strict.

The failure of the spectral mapping theorems for semigroups leads to a number of 
major difficulties in the semigroup theory. To discuss some of them
define the exponential growth bound of $(T(t))_{t \ge 0}$ by
\begin{eqnarray*}
\omega_0(T) :=\lim_{t \to \infty}\frac{\ln \|T(t)\|}{t},
\end{eqnarray*}
the spectral bound of its generator $A$ by
\begin{equation*}
s(A):=\sup \{{\rm Re} \, \lambda : \lambda \in \sigma(A)\},
\end{equation*}
and the pseudo-spectral bound of $A$ (or abscissa of uniform boundedness of the resolvent of $A$) by
\begin{eqnarray*}
s_0(A):=\inf \{\omega > s(A): R(\lambda, A)  \,\,
\text{is uniformly bounded for } {\rm Re}\, \lambda \ge \omega \}.
\end{eqnarray*}

The first two bounds possess  ``essential analogues'' given by
$$
\omega_e (T) :=\lim_{t \to \infty}\frac{\ln \|T(t)\|_{e}}{t} \qquad \text{and} \qquad
  s_e(A):=\sup \{{\rm Re} \, \lambda : \lambda \in \sigma_{e}(A)\}.
$$
Note that the first limit exists, and $r_{e}(T(t))=e^{\omega_e(T) t}$ for $t \ge 0,$
see e.g.\ \cite{Voigt}. 
It is also  crucial to observe that from \cite[Corollary IV.2.11]{EnNa00}
it follows  that $\omega_0(T)=\max\{s(A), \omega_e(T)\},$
and moreover
$\{\lambda \in \sigma (A): {\rm Re}\, \lambda > s_{e}(A)\}$ is an at most countable set
consisting of isolated eigenvalues of $A$ (with finite multiplicity).  Clearly,
\begin{equation*}
s_e(A)\le s(A) \le s_0(A) \le \omega_0(T) \qquad \text{and} \qquad s_e(A) \le \omega_e(T)\le \omega(T).
\end{equation*}
There are various examples of $C_0$-semigroups making all or some of the above
inequalities  strict, see e.g.\ \cite[Chapter 5.1]{ABHN}.
So that, in general,  neither the spectrum nor the resolvent of $A$ determine
the exponential norm bounds for  $(T(t))_{t \ge 0}.$
Semigroups with $s(A)<\omega_0(T)$ also arise from  concrete partial differential equations, such as damped 
wave equations, cf.\ Section~\ref{sec:exa}. 

On the other hand, if $X$ is a Hilbert space, then the well-known Gearhart-Herbst-Pr\"uss theorem 
guarantees that
\begin{equation}\label{type}
\omega_0(T)=s_0(A),
\end{equation}
and the exponential decay of $(T(t))_{t \ge 0}$ is equivalent to $s_0(A)<0.$ 
For an  exhaustive discussion of relations between these two and other related bounds see e.g.\
\cite[Chapters 5.1-5.3]{ABHN} or \cite[Chapters 1-4]{Ne96}.

To be able to obtain sharp lower bounds for the trajectories of \eqref{solaeq} we need to 
introduce the new resolvent bound $s_R(A)$ as the infimum of the set $S_R$ of  $a >s_e(A)$ satisfying
\begin{align}\label{defsr}
{\rm card}\, (\sigma_p(A)\cap (a+i\mathbb R))<\infty
\quad \text{and} \quad
\limsup_{|b|\to \infty }\|R(a+ib, A)\|<\infty.
\end{align}
 This bound will play a crucial role in the sequel. Note that every vertical line $a+i\mathbb R$ with $a> s_e(A)$
 contains at most countablly many eigenvalues.
Using the discreteness of the set  $\{\lambda \in \sigma (A): {\rm Re}\, \lambda > s_{e}(A)\}$
 and the Neumann series expansion for the resolvent,
one shows that the set $S_R$ is open. Thus the infimum is not attained, and we have 
\begin{align*}
 &s_R(A)=s_e(A) \qquad  \text {or} \qquad {\rm card}\, (\sigma_p(A)\cap (s_R(A)+i\mathbb R))=\infty\\
 &\text{or} \!\quad
 {\rm card}\, (\sigma_p(A)\cap (s_R(A) +i\mathbb R))< \infty \,\, \text{and} \,\, 
    \limsup_{|b| \to \infty} \|R(s_R(A) +ib, A)\|=\infty.
\end{align*}
Moreover, as the set  $\{\lambda \in \sigma (A): {\rm Re}\, \lambda > s_{e}(A)\}$
is discrete, one may also define $s_R(A)$ as the infimum
of the set of $a >s_e(A)$ such that
\begin{equation*}
\exists \, \, \beta=\beta(a) >0: \,\,  a+i(\mathbb R \setminus (-\beta,\beta))\subset \rho(A) \,\, \text{and} \,\, \sup_{|b|\ge \beta}\|R(a+i\beta, A)\|<\infty.
\end{equation*}

To explain the relevance of $s_R(A)$ and to relate it to the spectral properties of $(T(t))_{t \ge 0}$,
we introduce the notion of admissibility. We say that $\omega \in\RR$ is \emph{admissible}  if for every $t_0>0$, 
every  $\e>0$ and  every subspace $M\subset X$ of the form $M =\bigcap_{1\le  j \le n}{\rm Ker}\, y^*_j$
with $y^*_j \in D(A^*)$ for $j\in\{1,\dots n\}$, there exist $x\in M$ with $\|x\|=1$ and $\mu\in\CC$
with $\Re\mu=\omega$ such that
\begin{equation}\label{ineq_e}
\|T(t)x-e^{\mu t}x\|<\e, \qquad 0\le  t\le t_0.
\end{equation}
Note that $\codim M<\infty.$
Observe also that $x$ and $\mu$ depend on $t_0,$  $\epsilon$ and $M,$
and so $x$ is \emph{not}  an approximate eigenvalue of $T(t),$ $0\le t \le t_0,$  in general.
However, the notion of admissibility
will help us  to``emulate'' the approximate eigenvalues of $T(t)$ to an extent that
sufficices for the construction of growing solutions to nonlinear evolution equations.

Using Lemma \ref{codim} it is easy to show that $s_e(A)$ is admissible.
In this case, there exists $\mu$ with ${\rm Re}\, \mu=s_e(A)$ such that for any $\epsilon,t_0>0$ and
any  closed subspace of finite codimension $M$,  one can find a unit vector $x \in M$ satisfying
\eqref{ineq_e}.
However,  Theorem \ref{proplower} provides a more general statement showing that
in fact $s_R(A)$ is admissible.
Thus the next result is one of the basic tools in this paper.
\begin{theorem}\label{proplower}
Let $(T(t))_{t\ge 0}$ be a $C_0$-semigroup on a  Banach space $X$ with generator $A,$
and let $s_R(A)$ be defined by \eqref{defsr}. Then $s_R(A)$ is admissible.
\end{theorem}
The proof of the theorem is given in the appendix. It is similar to the proof
\cite[Proposition 4.4]{Muller_Tom}, though it is technically more demanding.

\section{Warm-up: initial observations and comments}\label{sec:warmup}

\subsection{Spectrum does not suffice}

It is well-known that the spectral radius is not continuous on $B(X).$ This leads 
to the fact that, in general, the instability of \eqref{solaeq} 
is not preserved under small Lipschitz perturbations.

For a fixed $a>0$ let $X= L^2(0, 2a)$ and consider the selfadjoint operator 
$(A f)(s)=sf(s)$ on $X.$
Clearly, $\sigma (A-a)=[-a,a]$ and thus there are initial
values $x_0 \in X$ for which the solutions of
\begin{equation}\label{selfadj}
x'(t)=(A-a)x(t),\qquad t \ge 0, \quad x(0)=x_0, 
\end{equation}
grow exponentially in the sense that $\|e^{(A-a)t} x_0\|\ge C_{x_0}e^{at/2}$ for $t\ge0$ with a
constant $C_{x_0}>0$.

By a classical result due to Herrero \cite{Herrero} (see \cite{Hadwin}
for a simple proof), if the spectrum of a  bounded normal operator on $X$ is connected and contains zero, then 
the operator is a limit of a sequence (or net) of bounded nilpotent operators on $X$ in the uniform operator 
topology. Hence there exists a family of nilpotent operators $(A_\e)_{\e >0}$ on $X$ such that 
$A_\e \to A$ in $L(X)$ as $\e \to 0.$ Consider the perturbed sysem
\begin{equation}\label{selfadj1}
x'(t)= (A-a)x(t) +(A_\e-A)x(t)= (A_{\e}-a)  x (t), \qquad t \ge 0.
\end{equation}
Note that for any  $\delta >0$ there exists $\e_0 >0$ such that 
\[
\sigma (A_{\e}-a)=\{-a\} \qquad \text{and}\qquad \|A_\e-A\|<\delta, 
        \qquad \e \in (0,\e_0).
\]
So, all solutions of \eqref{selfadj1} tend exponentially to $0$
though \eqref{selfadj1} only differs by an arbitrarily small Lipschitz perturbation
from the problem \eqref{selfadj} possessing exponentially growing orbits.
 This general phenomena  was observed in \cite{Slyusarchuk78}
and rediscovered recently in \cite{Sola_Rod}. However, \cite{Sola_Rod} went much further, see
Subsection \ref{rod} below. We add that in \cite{Slyusarchuk85}
a Lipschitz perturbation $K$ is constructed which destroys instability of a linear system
$x'(t)=Ax(t)$ and satisfies the strictly sublinear
growth assumption  $\|K(x)\|/\|x\| \to 0$ as $\|x\| \to 0.$
Our examples in Section~\ref{local_discr} differ from these treatments.

\subsection{Unstable orbits may co-exist with very stable ones}

On the Hilbert space $X=L^2((0,\infty), e^{-2t} dt)$ let the operator $Af=f'$ be defined on its 
maximal domain in $X$. Then $A$ generates the $C_0$-semigroup $(T(t))_{t \ge 0}$ of left shifts on $X,$ and 
$\sigma(A)=\{\lambda \in \mathbb C: {\rm Re}\, \lambda \le 2\}.$ While $(T(t))_{t \ge 0}$ is hypercyclic, 
i.e., its trajectories are dense in $X$ for a dense set of initial values (see e.g.\ \cite{Desch}),
the trajectories  vanish eventually for every initial value with compact support. 
Thus, we have a dense set of initial values $x_0$ for the abstract Cauchy problem
\[
x'(t)=Ax(t), \qquad x_0=x_0,
\]
yielding  ``superstable'' mild solutions and, at the same time, a dense set of initial values $x_0$ 
whose trajectories have a totally unstable behavior. For more information on hypercyclic semigroups,
we refer to e.g.\ \cite{Desch}.

\subsection{Sublinear perturbations}\label{rod}

 Rodriguez and Sol\'a-Morales  proved in \cite{Sola_Rod} that in an infinite-dimensional separable real 
Hilbert space $X$ there exists a $C^1$ map $T : X \mapsto X$ of the form $T = A+K,$ where $A$ is a bounded
linear operator on $X$ with the spectral radius larger than one
and the nonlinearity $K$ satisfies
 $\|K(x)\|/\|x\| \to 0$ as $x \to 0$ and  $K(0)=0.$ Nevertheless
the fixed point $x = 0$ of $T$ is (exponentially) stable. Moreover, given $c_2>c_1>0$
there exists $T$ as above such that
\[
\left(\frac{1}{-\log\|x\|}\right)^{c_2}
  < \frac{\|Kx\|}{\|x\|}
 < 4\left(\frac{1}{-\log(\|x\|)}\right)^{c_1}
\]
for $x$ with $\|x\|$ small enough.

\subsection{Rank-one perturbations}

The instability properties for \eqref{solaeq1} become rather arbitrary if the essential spectrum of $A$ is merely 
contained in the closed unit disc. A good illustration for that phenomena is provided by the fact that there exist 
a unitary operator $A$ on a Hilbert space $X$ and rank-one operator $K$ on $X$ such that the operator $A+K$ 
is hypercyclic, see e.g.\ \cite{Grivaux}, \cite{Baranov}. Clearly, $r_{e}(A+K)=r_{e}(A)=1.$ On the other hand,
it is well-known that there exist a unitary operator $A_1$ and a rank-one operator $K_1$ such that $A_1+K_1$
is strongly stable, that is, $(A_1+K_1)^n x \to 0$ as $n \to \infty$ for every $x \in X.$
Indeed, if  $S$ is a unilateral shift on the Hardy space $H^2(\mathbb D),$
where $\mathbb D$ is the open unit disc,
then \cite{Clark} shows that certain unitary operators on $H^2(\mathbb D)$
arise as rank-one perturbations of compressions of $S$
to the closed invariant subspaces of $S^*.$ Moreover, clearly  $S^{*n} \to 0$ as $n \to \infty,$ strongly.
A more general set-up for such a construction can be found in \cite{Nakamura}.

Apparently, a similar example can be constructed in a continuous framework.
Recall that for any selfadjoint operator $A$ on a Hilbert space $X$ the semigroup $(e^{iAt})_{t \ge 0}$ is unitary. It is plausible that  there exist bounded selfadjoint operators $A$ and $A_1$  on $X$ and rank-one perturbation $K$ and $K_1$
such that semigroup $(e^{i (A+K)t})_{t \ge 0}$ is hypercyclic,  while the semigroup $(e^{i(A_1+K_1)t})_{t \ge 0}$ is
strongly stable. Although, no results have been published in this context.

\section{Growth and instability for discrete time}\label{discrete}

In this section, we assume that $X$ is a Banach space, $A\in B(X)$, and we
let $(K_n)_{n=1}^{\infty}$ be a sequence of compact maps $K_n: X\to X$
(in general, non-linear). For $x\in X$ and $n\in\NN$ let
\begin{equation}\label{fn}
f_n(x)=(A+K_n) \cdots (A+K_1)x,
\end{equation}
and set  $f_0(x)=x$. In other words, $(f_n(x))_{n=0}^{\infty}$ is a solution of the difference equation
\begin{equation}\label{fn1}
x_{n+1}=(A+K_n) x_n, \qquad x_0=x.
\end{equation}

The next simple lemma replaces the difference equation \eqref{fn1} by a ``difference inclusion''.
In this way, arguing up to compact sets,
we will able to reduce the study of asymptotic properties of $(f_n)_{n=1}^\infty$ to the study of such properties
for linear iterations $(A^n)_{n=1}^\infty.$

\begin{lemma}\label{lem1}
Let $(f_n)_{n=1}^{\infty}$ be given by \eqref{fn},  $n\in\NN$ be fixed, and
$X_0\subset X$ be a bounded set. Then there exists a compact set $C\subset X$ such that
$$
f_j(x)\in A^jx+C
$$
for all $x\in X_0$ and $j\in\{1,\dots,n\}$.
\end{lemma}

\begin{proof}
It is easy to see by induction on $j$ that
$$
f_j(x)=A^jx+\sum_{s=0}^{j-1} A^{j-s-1}K_{s+1} f_{s}(x)
$$
for each $x\in X$ and each $j\in\{1, \dots, n\}$. Using induction again, we infer that
for every $s\in \{1, \dots, n\}$ the set $f_{s}(X_0)$ is  bounded, and thus
$K_{s+1}f_{s}(X_0)$ is precompact. 
Therefore, the set
$$\bigcup_{j=1}^n\sum_{s=0}^{j-1}A^{j-s-1}K_{s+1}f_{s}(X_0)$$
is precompact, and we let $C$ be its closure.
Then $f_j(x)\in A^jx+C$ for all $x\in X_0$ and $j\in\{1, \dots, n\}$.
\end{proof}

Now we are able to give a lower bound for a finite piece of the trajectory
in terms of the deviation of its linear part from a finite-dimensional subspace.
The next simple lemma provides an intuition behind the crucial Lemma \ref{lem3}.

\begin{lemma}\label{lem2}
Let $(f_n)_{n=1}^{\infty}$ be given by \eqref{fn}. Let $n\in\NN$ be fixed and let $X_0\subset X$ be a bounded set.
Then for every $\e>0$ there exists a finite-dimensional subspace $F\subset X$ such that
$$
\|f_n(x)\|\ge\dist\{A^nx,F\}-\e
$$
for all $x\in X_0$.
\end{lemma}

\begin{proof}
Let $C$ be the compact set constructed in Lemma \ref{lem1}. Since $C$ is compact, there exists a finite-dimensional 
subspace $F\subset X$ such that $\dist\{c,F\}<\e$ for every $c\in C$. For each $x\in X_0$ we have
$f_n(x)\in A^nx+C$ implying
\[
\|f_n(x)\|\ge\dist\{A^nx,C\}\ge \dist\{A^nx,F\}-\e.\qedhere
\]
\end{proof}

Having obtained the estimate for a finite piece of trajectory, we can spread it out to a finite-dimensional 
subspace of $X$, as the next lemma shows.

\begin{lemma}\label{lem3}
Let $A \in B(X),$ $n\in\NN$, $\e > 0,$ and let $F\subset X$ be a finite-dimensional
subspace. Then the following assertions hold.
\begin{itemize}
\item [(i)] There exists a unit vector $u\in X$ such
that $\dist \{A^n u, F\} \ge \frac{1-\e}{2} \|A^n\|_\mu$.
\item [(ii)] There exists a unit vector $u\in X$ such that
$\dist\{A^ju, F\}\ge\frac{1-\e}{2}r_{e}(A)^j$ for all $j=1,\dots,n$.
\end{itemize}
\end{lemma}
\begin{proof}
To prove (i), note that  by Lemma \ref{geomlemma} in the appendix there exists a closed
subspace $M\subset X$ of a finite codimension such that
$$
\|f+m\|\ge (1 - \e/2)\max\{\|f\|,\|m\|/2\}
$$
for all $f\in F$ and $m\in M$. Since $L :=  A^{-n}(M)$ is a closed subspace of a finite
codimension, there exists a unit vector $u\in L$ with $\|A^nu\|\ge (1 - \e/2)\|A^n\|_\mu$. Then
$$
\dist \{A^n u, F\} = \inf\{\|A^n u + f \|: f\in F \} \ge \frac{1 - \e/2}{2}
\|A^n u\|\ge  \frac{1 - \e}{2}
\|A^n\|_\mu,
$$
which gives (i).

Let $\mu\in\sigma_{e}(A)$ with $|\mu|=r_{e}(A)$. Then $\mu\in\partial\sigma_{e}(A)$, and
by Lemma~\ref{codim} 
there exists a unit vector $u\in L$ with $\|(A-\mu)u\|<\delta$.
For each $j=1,\dots,n$ we have
$$
\|(A^j-\mu^j)u\|=\Bigl\|\sum_{k=0}^{j-1}A^{j-k-1}\mu^k (A-\mu) u\Bigr\|\le j\|A\|^j\delta.
$$
We now choose $M$ as in the proof of (i) and set $L=\bigcap_{j=1}^nA^{-j}M$. Then there is a unit vector $u\in L$ 
such that $\|(A-\mu)u\|$ is so small that $\|(A^j-\mu^j)u\|\le r_{e}(A^j)\e/2$ for all $j=1,\dots,n$. Thus
$\|A^ju\|\ge (1-\e/2)r_{e}(A^j)$, and as above we deduce that $\dist\{A^ju, F\}\ge\frac{1-\e}{2}r_{e}(A^j)$
for all $j=1,\dots,n$.
\end{proof}

The above lemmas allow us to obtain  exponential lower bounds
for the norms of trajectories $\|f_{n}(x)\|$ for residual set of initial values
$x$  along some subsequences $\{n_k\}.$
Apart from a usual Baire's category theorem,
the geometrical Lemma \ref{lem3} is indispensable here.

\begin{theorem}\label{subeq}
 Let $A \in B(X)$ and $K_n:X\to X$ be compact for $n\in \NN$.
 Take a sequence $(a_n)_{n=1}^{\infty}$  in  $\RR_+$  such that $\lim_{n\to\infty} a_n = 0$,
 and let $L\subset\NN$ be infinite. For $x \in X$ define $f_n(x)$ by \eqref{fn}. Then the set
$$
\bigl\{x\in X : \, \hbox{there are infinitely many} \,\, n\in L \,\,\hbox{with} \,\, \|f_n(x)\|\ge
a_n\|A^n\|_\mu\bigr\}
$$
is residual in $X$.
\end{theorem}
\begin{proof}
The statement is clear if $\|A^n\|_\mu = 0$ for an infinite number of $n\in L$.
Thus, without loss of generality, we may assume that $\|A^n\|_\mu > 0$ for all $n\in L$. For $k\in\NN$ let
$$
M_k = \{x\in X : \hbox{ there exists }\,\, n\in L \text{ \ with \ } n \ge k \, \hbox{ and }\, 
  \|f_n(x)\| > a_n\|A^n\|_\mu\}.
$$
Clearly $M_k$ is an open set. We show that it is dense in $X$.
Let $y\in X$ and fix $\e > 0$. Choose $n\in L$ with $n\ge k$ and $a_n < \e/4$. By Lemma \ref{lem1}, there
exists a compact set $C\subset X$ such that
$$f_n(y+v)\in A^n(y + v) + C$$
for all $v\in X$ with $\|v\|\le \e$. Since $C$ is compact, there exists a finite-dimensional subspace $
F\subset X$ such that $\dist \{c, F\} < \frac{\e\|A^n\|_\mu}{12}$ for all $c\in C$.
Lemma \ref{lem3} (i) then yields a unit vector $u\in X$ with $\dist \{A^nu, F\} >\frac{\|A^n\|_\mu}{3}$. Note that
$$
\dist \{A^n(y + \e u), F\} + \dist \{A^n(y - \e u), F\} 
 \ge 2\e\dist \{A^nu, F\}
 \ge \frac{2\e}{3} \|A^n\|_\mu.
$$
Indeed, passing to the quotient space $X/F$, the former inequality follows from
$$
\|\pi(A^n y) +\e \pi(A^n u)\|_{X/F}+\|\pi(A^n y)- \e \pi(A^n u)\|_{X/F} \ge \frac{2\e}{3} \|\pi(A^n u)\|_{X/F}
$$
via  the triangle inequality, where $\pi$ denotes the corresponding quotient mapping.
So  $x:= y + \e u$ or $x := y - \e u$ satisfy
$$\|x - y\| \le\e \quad \text{and} \quad \dist \{A^nx, F\}\ge \frac{\e}{3}\|A^n\|_\mu.$$
We conclude
$$
\|f_n(x)\|
\ge
\dist \{A^nx, F\} - \frac{\e\|A^n\|_\mu}{12}
\ge
\frac{\e\|A^n\|_\mu}{4}\ge a_n\|A^n\|_\mu.
$$
Hence $x\in M_k$, and $M_k$ is dense since the choice of $y$ and $\e$ was arbitrary.
By the Baire category theorem, $\bigcap_{k=1}^\infty M_k$ is a
dense $G_\delta$ set and thus residual.
\end{proof}

\begin{corollary}\label{cor2}
Let $A\in B(X)$ satisfy $\sup\{\|A^n\|_\mu:n\in\NN\}=\infty$. If $(f_n)_{n=1}^\infty$ is as in Theorem \ref{subeq},
then the set
$$
\big\{x\in X: \sup\nolimits_{n\in\NN}\|f_n(x)\|=\infty\big\}
$$
is residual. In particular, this is true if $r_{e}(A)>1$.
\end{corollary}

If we concentrate on merely dense sets of initial values rather than residual ones,
then we can construct orbits $(f_n(x))_{n=1}^{\infty}$ with  exponential growing norm lower bounds
for \emph{all} large $n$.

\begin{theorem}\label{theorem1}
Let $A \in B(X)$, $(K_n)_{n=1}^\infty$ be a sequence of compact maps on $X$, and $(f_n)_{n=1}^\infty$ be given by 
\eqref{fn}. Take  a non-increasing sequence $(a_n)_{n=1}^{\infty}\subset \mathbb R_+$ satisfying
$\lim_{n\to\infty}a_n=0$. Fix $y\in X$, $r>0$, and  $n_0\in \mathbb N$ with
$a_{n_0}<\frac{r}{2}$. Then there exists $x\in \overline{B}(y,r)$ such that
$$
\|f_n(x)\|\ge a_n r_{e}(A)^n, \qquad n\ge n_0.
$$
\end{theorem}
\begin{proof}
Denote $r_{e}:=r_{e}(A)$ for shorthand. 
The statement is clear if $r_{e}=0$.
Let $r_{e}>0$. We start with several convenient reductions.
Without loss of generality we may assume that $r_{e}=1.$ If not, then consider the operator
$r_{e}^{-1}A$ and the compact mappings $ r_{e}^{-n} K_n(r_{e}^{n-1}\cdot )$ for $n \ge 1$.
Replacing $A$ with $\lambda A$ for some $|\lambda|=1$ and $K_n$ with $\lambda^{n}K_n (\lambda^{-n+1}\cdot)$
 if necessary, we may assume that $1\in\sigma_{e}(A)$.
The operator $A-I$ is not upper Fredholm since $1\in\partial\sigma_{e}(A),$  see Lemma~\ref{codim}.

Proposition~\ref{Sch} thus gives a compact linear operator
$\widetilde K$ on $X$ with $\dim{\rm Ker}\, (A-I-\widetilde K)=\infty$.
Thus, passing to $A-\widetilde K$ and $K_n+\widetilde K$ if necessary, we may assume that 
$\dim{\rm Ker}\, (A-I)=\infty.$

Take some $c_1\in(2a_{n_0},r)$. Fix integers $n_0<n_1<n_2<\cdots$ such that 
$a_{n_k}<2^{-(k+2)}(r-c_1)$. Set 
\[c_k=2^{-(k-1)}(r-c_1), \qquad k\ge 2.
\]
 Choose positive numbers $\e_k$ such that $\e_{k+1}<\e_k$ and
$$\frac{(1-\e_k)^2}{2}c_k-\e_k\ge a_{n_{k-1}}$$ for all $k \in \mathbb N$.

Set $x_0=y$, $F_0=\{0\}$ and $M_0=X$. Inductively, we construct finite-dimensional subspaces 
$F_1\subset F_2\subset\cdots$, finite-codimensional closed subspaces $M_1\supset M_2\supset\cdots$, and
unit vectors $x_k\in M_k\cap{\rm Ker}\,(A-I)$ for $k\ge1$ such that \eqref{eq:ind} below is true.

Let $k\ge 1$ and suppose that the vectors $x_0,\dots, x_{k-1}$ and spaces $F_0,\dots,F_{k-1}$ and 
$M_0,\dots, M_{k-1}$ have already been constructed. Lemma \ref{lem1} provides a compact set 
$C_k\subset X$ such that 
$$f_n(u)\in A^nu+C_k$$ 
for all $n\le n_k$ and $u\in X$ with $\|u\|\le\|y\|+r.$
Using the compactness of $C_k,$ we find a finite-dimensional subspace $F_k\supset F_{k-1}$ such that 
\begin{equation}\label{eq:ind}
\{y,Ay,\dots,A^{n_k}y,x_1,\dots,x_{k-1}\}\subset F_k \quad  \text{ and } \quad \dist\{d,F_k\}\le\e_k
\end{equation} 
 for all $d\in C_k$. Lemma~\ref{geomlemma} yields
 a closed subspace $M_k\subset M_{k-1}$ of finite codimension such that
$$
\|f+m\|\ge (1-\e_k)\max\Bigl\{\|f\|,\frac{\|m\|}{2}\Bigr\}
$$
for all $f\in F_k$ and $m\in M_k.$ Since ${\rm Ker}\, (A-I)$ has infinite dimension,
there exists a unit vector  $x_k\in M_k\cap{\rm Ker}\,(A-I)$.

For the elements  $x_k,$ $k\in\NN,$ constructed above, set 
$$x=y+\sum_{j=1}^\infty c_jx_j.$$
We show that $x$ satisfies the assertions. First note that
$$\|x-y\|\le\sum_{j=1}^\infty c_j=c_1+\sum_{j=2}^\infty \frac{r-c_1}{2^{j-1}}=r,$$
so that $x \in \overline{B}(y,r).$ Let $k\ge 1$ and $n_{k-1}\le n\le n_k$. We estimate
\begin{align*}
\|f_n(x)\| &\ge\dist\{A^nx,C_k\}\ge\dist\{A^nx,F_k\}-\e_k\\
  &=\dist\Bigl\{\sum_{j=k}^\infty c_jx_j,F_k\Bigr\}-\e_k,
\end{align*}
employing that
$$A^nx=A^ny+\sum_{j=1}^\infty c_jx_j \quad \text{and} \quad A^ny,x_1,\dots,x_{k-1}\in F_k.$$
For $j\ge k$, the vector $x_j$ belongs to $M_j\subset M_k$, and so
$$
\|f_n(x)\|\ge \frac{1-\e_k}{2}\cdot\Bigl\|\sum_{j=k}^\infty c_jx_j\Bigr\|-\e_k.
$$
Since $x_k\in F_{k+1}$ and $x_j\in M_j\subset M_{k+1}$ for $j\ge k+1$, we obtain
\begin{align*}
\|f_n(x)\| &\ge \frac{(1-\e_k)(1-\e_{k+1})}{2}\|c_kx_k\|-\e_k\\
  &\ge \frac{(1-\e_k)^2}{2}\cdot c_k -\e_k\ge a_{n_{k-1}}\ge a_n.
\end{align*}
Hence
$$
\|f_n(x)\|\ge a_n
$$
for all $n\ge n_0$, and the statement follows.
\end{proof}
\begin{remark}
If $X$ is a Hilbert space, then in the proof of  Theorem \ref{theorem1} one can take $M_k=F_k^\perp$ in each step. 
It is then possible to obtain a slightly better estimate: Let $y\in X$, $n_0\ge 0$, and $a_{n_0}<r$. Then there 
exists $x\in X$ with $\|x-y\|\le r$ and $\|f_n(x)\|\ge a_nr_{e}(A)^n$ for all $n\ge n_0$.
\end{remark}

\begin{corollary}
Under the assumptions of Theorem~\ref{theorem1},
let $(a_n)_{n=1}^{\infty}\subset \mathbb R_+$ be a non-increasing sequence satisfying $\lim_{n\to\infty}a_n=0$. Then
\begin{itemize}
\item [(i)] there exists $x\in X$ such that $\|f_n(x)\|\ge a_nr_{e}(A)^n$ for all $n\in\NN$;
\item [(ii)] there exists a dense subset of vectors $x\in X$ such that $\|f_n(x)\|\ge a_nr_{e}(A)^n$ for all $n$
     sufficiently large.
\end{itemize}
\end{corollary}
\begin{proof}
To prove (i), it suffices to set $y=0$ and to choose a big enough radius $r$ in Theorem \ref{theorem1}.

 Let now $y\in X$ and $\e>0$ be fixed. Let $n_0$ be such that $a_{n_0}<\frac{\e}{2}$. By Theorem \ref{theorem1}, 
 there exists $x\in X$ such that $\|x-y\|\le\e$ and $\|f_n(x)\|\ge a_nr_{e}(A)^n$ for all $n\ge n_0$, and
 assertion (ii) follows.
\end{proof}

\begin{remark}\label{slyusar}
If one is interested in merely local instability properties of \eqref{fn1} assuming $r_e(A)>1,$ then  results close in spirit
can be found in \cite{Slyus1}.
\end{remark}

\section{Local instability for discrete time: counterexample}\label{local_discr}

As we mentioned in the introduction, it is natural  to try to combine ``metric'' instability results, such as e.g.\ in \cite{dalec}, \cite{Henry}, \cite{Shatah1}, \cite{Gallay}, and  ``topological'' instability conditions as above, in order to obtain a local result showing instability for the zero solution to
the difference equation
\begin{equation}\label{sum}
x_{n+1}=(A+K+G) x_{n}, \qquad n \in \mathbb N, \quad x_0=x,
\end{equation}
in a Banach space $X.$ Here $A$ is a linear operator on $X$ with $r_{e}(A)>1,$ $K$ is a compact operator
on $X$ with $K(0)=0$, and  $G$ is a continuous map defined on a ball
$B(0,r)\subset X$ such that $\|G(x)\|=O(\|x\|^{1+\delta})$ as $\|x\| \to 0$ for some $ \delta >0$.
Unfortunately, this is not possible, in general, even if $X$ is a Hilbert space, as Example \ref{example2} shows.
Recall that the zero solution of \eqref{sum} is called stable  if for every $\e > 0$ there is   $\delta > 0$ such
for all $x_0$ with $\|x_0\|<\delta$ one has $\sup_{n}\|x_n\|<\e.$

We first note a simple lemma.
\begin{lemma}\label{lemma_fg}
Let $\e \in(0,2^{-7})$. Then there exist continuous functions $f:[0,\infty)\to[0,\e/2]$ and 
$g:[0,\infty)\times[0,\infty)\to [0,2]$ such that
\begin{itemize}
\item [(i)]\,\, $f(s)=0, \qquad0\le s\le \e^3\quad \hbox{or} \quad s\ge\e$;

\item [(ii)] \,\, $f(s)=\e/2, \qquad 4\e^3\le s\le\frac{\e}{2}$;
\end{itemize}
and
\begin{itemize}
\item [(iii)] \,\, $g(s,t)=0, \qquad 0\le s\le \e^3\quad \hbox{or} \quad s\ge\e,\, t\ge 0,$;

\item [(iv)] \,\, $g(s,\frac{\e}{2})=2, \qquad 4\e^3\le s\le 16\e^3$;

\item [(v)] \,\, $g(s,t)\le s+\frac{t^2}{s}, \qquad s> 0, \, t \ge 0$.
\end{itemize}
\end{lemma}
\begin{proof}
The existence of $f$ is obvious.

Let $\widetilde g:[0,\infty)\times[0,\infty)\to [0,2]$ be any continuous function satisfying (iii) and (iv). 
Set $g(s,t)=\min\{\widetilde g(s,t), s+\frac{t^2}{s}\}$ for $s> 0$ and $t \ge 0,$ and $g(0,t)=0$ for $t\ge0$.
Clearly, $g$ fulfills (iii) and (v).
If $4\e^3\le s\le 16\e^3$ and $t=\frac{\e}{2},$ then
$$
s+\frac{t^2}{s}\ge \frac{t^2}{s}\ge \frac{\e^2}{4\cdot 16\e^3}=\frac{1}{64\e}> 2= \tilde{g}(s,t).
$$
So $g$ satisfies (iv).
\end{proof}

Next we provide a counterexample for local \emph{instability} of \eqref{sum}, where the nonlinear part
is given by the sum of two nonlinear parts, each guaranteeing local instability when considered separately.
To this aim, we will need an auxiliary example showing that there are a Hilbert space $H$ and
operators $A$, $K$ and $G$ on $H$ as above such that the zero solution to \eqref{sum} is stable
for fixed $\epsilon >0.$

\begin{example}\label{example1}
Let $H_0$ be an infinite-dimensional Hilbert space, and set $H=H_0\oplus \mathbb C$ with the Hilbert space norm. 
Let $\e, f$ and $g$ be given as above. Consider the mappings $A\in B(H),$ $K:H \to H,$ and $G:H\to H$ defined by
\begin{align*}
A(y,a)&:=(2y,0), \qquad (y,a) \in H;\\
K(y,a)&:=(0,f(\|y\|)), \qquad (y,a) \in H;\\
G(y,a)&:=\bigl(-g(\|y\|,|a|) y,0\bigr), \qquad (y,a) \in H.
\end{align*}

It is clear that $A\in B(H)$, $r_{e}(A)=2$, $K$ is compact, and $G$ satisfies
\begin{align*}
\|G(y,a)\| &=g\bigl(\|y\|,|a|\bigr)\cdot\|y\|
   \le \Bigl(\|y\|+\frac{|a|^2}{\|y\|}\Bigr)\|y\|\\
  &= \|y\|^2+|a|^2=\|(y,a)\|^2.
\end{align*}
Moreover,
$$(A+K+G)(y,a)=\Bigl((2-g\bigl(\|y\|,|a|)\bigr)y, f(\|y\|)\Bigr).$$

Let $x_0=(y_0,a_0)\in H$ be such that $\|y_0\|\le\e^3$ and $|a_0|\le\e^3$,
and let $(x_{n})_{n=1}^\infty$ be given by \eqref{sum}.
For $n \ge 1$ denote $x_{n}: =(y_n,a_n).$ First observe that $|a_n|\le\frac{\e}{2}$ for all $n \ge 1,$
by the definition of $f$. We prove that
\begin{equation}\label{yn}
\|y_n\|\le \epsilon/2, \qquad n\in \mathbb N.
\end{equation} 
Observe that $\|y_n\|\le 2\|y_{n-1}\|,$ $n \in \mathbb N$. We distinguish two cases.

A) There exists $n$ such that $\|y_n\|> 4\e^3$ and $\|y_{n+1}\|> 4\e^3$. Let $n_0$ be the smallest integer 
with this property. We then obain
$$\|y_{n_0}\|\le 2\|y_{n_{0}-1}\|\le 8\e^3,$$  and so
$$a_{n_0+1}=f(\|y_{n_0}\|)=\frac{\e}{2} \quad \text{and} \quad 4\e^3\le\|y_{n_0+1}\|\le 2\|y_{n_0}\|\le 16\e^3.$$
 Hence, $g\bigl(\|y_{n_0+1}\|,|a_{n_0+1}|\bigr)=2$ and $y_{n_0+2}=0$. Thus $y_k=0$ for all $k\ge n_0+2$.
Moreover, $\min\{\|y_m\|,\|y_{m+1}\|\}\le 4\e^3$ for all $m\le n_0-1$. It follows
$$\text{either}\qquad  \|y_{m+1}\|\le 4\e^3\qquad \text{or} \qquad \|y_{m+1}\|\le 2\|y_m\|\le 8\e^3.$$
We have shown $\sup_k\|y_k\|\le 16\e^3$ in this case.

B) Let $\min\{\|y_n\|,\|y_{n+1}\}\le 4\e^3$ for all $n$. We again have
$$\text{either}\qquad \|y_{n+1}\|\le 4\e^3\qquad \text{or}\qquad \|y_{n+1}\|\le 2\|y_n\|\le 8\e^3,$$
which yields $\sup_n\|y_n\|\le 8\e^3$.

In both cases we have shown $\sup_n\|y_n\|\le 16\e^3\le\frac{\e}{2}$, and so $\sup_n\|x_n\|\le\e$.
 \end{example}

Now we refine the construction in Example \ref{example1} and make it work for all $\epsilon>0$
thus obtaining the desired counterexample.
\begin{example}\label{example2}
Let $(\e_j)_{j=1}^\infty$ be a fast decreasing sequence of positive numbers such that $\e_1<2^{-7}$ 
and $\e_{j+1}<\frac{\e_j^3}{4}$ for all $j\ge 1$. As above, let $H=H_0\oplus\ell^2$ where $\dim H_0=\infty$. 
Take functions $f_j$ and $g_j$ as defined in Lemma \ref{lemma_fg} for $\e=\e_j$. 
Define mappings $A,$ $K,$ and $G:H\to H$ by
\begin{align}
A\bigl(y,(a_j)_{j=1}^\infty\bigr)&=\bigl(2y,(0)\bigr),
 \label{defin1}\\
K\bigl(y,(a_j)_{j=1}^\infty\bigr)&= \bigl(0, (f_j(\|y\|))_{j=1}^{\infty}\bigr),\label{defin2} \\
G\bigl(y,(a_j)_{j=1}^\infty\bigr)&=\Bigl(-y\sum_{j=1}^\infty g_j\bigl(\|y\|,|a_j|\bigr) ,0\Bigr) \label{defin3}
\end{align}
for all $\bigl(y,(a_j)_{j=1}^{\infty}\bigr) \in H.$ Note that for every $(y,(a_j)_{j=1}^\infty)\in H$ there is at most one $j$ such that $g_j\bigl(y,|a_j|\bigr)\ne 0$.
So $G$ is well-defined. Similarly, there is at most one $j$ with $f_j(\|y\|)\ne 0$.

Clearly, $A\in B(H)$, $r_{e}(A)=2$ and $K$ is compact since
$$
K(H)\subset \Bigl\{(0,(a_j)_{j=1}^\infty): |a_j|\le\frac{\e_j}{2}\Bigr\}.
$$
Furthermore,
\begin{align*}
\|G(y,(a_j)_{j=1}^\infty)\|\le&
\|y\|\cdot \sum_{j=1}^\infty g_j\bigl(\|y\|,|a_j|\bigr)=
\|y\|\cdot\max_j g_j\bigl(\|y\|,|a_j|\bigr)\\
\le& \max_j\bigl(\|y\|^2+|a_j|^2\bigr)\le\|(y,(a_j)_{j=1}^\infty)\|^2.
\end{align*}
We have
$$
(A+K+G)\bigl(y,(a_j)_{j=1}^\infty\bigr)=
\Bigl((2-\sum_{j=1}^\infty g_j\bigl(\|y\|,|a_j|)\bigr)y,\bigl(f_j(\|y\|)\bigr)_{j=1}^\infty\Bigr).
$$

Now with $A$, $G$ and $K$ defined by \eqref{defin1}-\eqref{defin3}, we consider equation \eqref{sum},
and prove that its zero solution is stable.
Let  $x_0 \in H$ and $k \in \mathbb N$  be such that
$\|x_0\|\le\frac{\e_k^3}{4}.$
and let $(x_n)_{n=1}^\infty$ be given by \eqref{sum}.

Write $x_n=\bigl(y_n, (a_{n,j})_{j=1}^\infty \bigr)$ for all $n \ge 0.$
We show first that $\|y_n\|\le\frac{\e_k}{2}$ for all $n$.
Suppose the contrary. Let $n_0$ be the smallest integer with $\|y_{n_0}\|>\frac{\e_k}{2}$. Let $n_1$ 
be the largest integer with $0\le n_1<n_0$ and $\|y_{n_1}\|<\frac{\e_k^3}{2}$. Note that
$$
0\le\sum_{j=1}^\infty g_j\bigl(\|y\|,|a_j|\bigr)\le 2
$$
for all $\bigl(y,(a_j)_{j=1}^\infty\bigr)\in H$, and hence
$$
\|y_{n+1}\|\le 2\|y_n\|\
$$
for all $n$. We infer that
$$\|y_{n_1}\|\ge\frac{\e_k^3}{4},\quad \text{and so}\quad
\frac{\e_k^3}{4}\le\|y_m\|\le\frac{\e_k}{2}$$ for all $m=n_1,\dots,n_0-1$,
implying
$$f_j(\|y_m\|)=0=g_j(\|y_m\|,|a_j|)$$
for all $j\ne k$ and $m=n_1,\dots,n_0-1$.
We further obtain
$$\|y_{n_1+1}\|\le 2\|y_{n_1}\|<\e_k^3 \quad \text{and} \quad a_{n_1+1,k}=f_k(\|y_{n_1}\|)=0.$$

Consider the orthogonal projection $P:H\to H$ onto $H_0\oplus\CC$ defined by
$$
P\big(y, (a_j)_{j=1}^\infty\big)=(y,a_k).
$$
By the above observations,   $Px_{n_1+1}=(y_{n_1+1},0)$ satisfies
$$\|Px_{n_1+1}\|\le\e_k^3$$
 and $Px_{n_1+1}, Px_{n_1+2},\dots,Px_{n_0}$ are the iterations described
in the previous example for $\e=\e_k$. Using the estimate \eqref{yn}, we infer that 
$\|y_{n_0}\|\le \frac{\e_k}{2},$
a contradiction. We have shown that
$$\sup_n\|y_n\|\le\frac{\e_k}{2}\quad \text{and} \quad
\sup_{j,n}f_j(\|y_n\|)\le\frac{\e_k}{2},$$ 
so that
$\sup_n\|x_n\|\le\e_k.$ Now to prove the stability of the zero solution to \eqref{sum}, given $\epsilon>0,$
it suffices to consider $\epsilon_k \in (0,\epsilon),$ and to choose $\delta=\epsilon_k^3/4.$
\end{example}

We are short of a similar construction for continuous time. However, we suspect
that an example similar to the above  
for \eqref{solaeq} may not exist.
 
\section{Well-posedness for continuous time and global existence}\label{sec:wp}
We are primarily interested in nonlinear evolution equations of the form
\begin{equation}\label{ca1}
x'(t)=A x(t)+ K(t, x(t)), \qquad x(0)=x_0, \quad t\in [0,T),
\end{equation}
where $T\in(0,\infty]$, $A$ generates a $C_0$-semigroup $(T(t))_{t \ge 0}$ on a Banach space $X$, and
$K:[0,\infty)\times X \to X$ is a function such that $x\mapsto K(t,x)$ is compact for fixed $t$ and that
satisfies some additional regularity assumptions.
One can study classical solutions of \eqref{ca1}, that is, maps
$x \in C^1([0,T), X)\cap C([0,T), D(A))$ solving \eqref{ca1}.
However, it is often convenient to deal with the integrated version of \eqref{ca1}, namely
\begin{equation}\label{ca_problem}
x(t)=T(t)x_0+\int_0^t T(t-s)K(s, x(s)) \, ds, \qquad t \ge 0.
\end{equation}
Recall that a continuous function $x: [0,T)\to X$ is called a \emph{mild} solution of \eqref{ca1} 
if it satisfies \eqref{ca_problem}.
Each classical solution of \eqref{ca1} is a mild solution, but not conversely. A mild solution is a classical 
one if $x_0 \in D(A)$ and, for instance, 
$K \in C^1([0,T)\times X,X).$
 In the sequel we only treat mild solutions, which we
will call \emph{solutions} from now on, for simplicity.

If $X$ has nice geometric properties, e.g., if it is reflexive, then mild solutions can be identified with
so-called weak solutions of \eqref{ca1}, that is, $x \in L^1([0,T), X)$ satisfying \eqref{ca1}
in a weak sense. This notion   will not be  used in this paper, however.

Let the solution $x(t)=x(t,x_0)$ of \eqref{ca_problem} exist for all $t\ge0,$ i.e., one has $T=\infty.$
It is called (nonlinearly) stable  if for every $\e > 0$ there is radius   $\delta > 0$ such
for all  $y_0 \in B(x_0,\delta)$ all solutions $y$ with $y(0) = y_0$ are defined  on $[0,\infty)$
and satisfy $\|y(t)-x (t)\|<\e$ for all $t\ge0$.
Without loss of generality, one may assume here that $x_0=0.$
If the solution $x$ is not stable, then it is said to be unstable.

Instability is a local and rather weak property. It is often of interest to show \emph{global} properties 
of $x$  which are stronger than mere instability. Such properties are the main topic of this paper.  
To not overshadow the study of asymptotics of the solutions to \eqref{ca_problem} with assumptions 
on its well-posedness, we will just postulate the  existence and minimal regularity properties of solutions
that we need in the sequel. They can be satisfied in many situations of interest as we will make clear below.
Aiming at the long-term behavior, it is natural to consider a set-up
when the solutions of \eqref{ca_problem} exist globally, i.e., on the whole of $[0,\infty).$ However, we do not need
uniqueness of solutions in our main Theorem~\ref{main}. Below we will mention several statements which could be used
as additional assumptions to our asymptotic results, so that the results could be formulated
in a priori terms.

The next local existence result is well-known, see \cite[p. 343-345, p. 350]{Segal} or e.g. \cite[Section 4]{Cazenave},
\cite[Section 6]{Pazy}.
\begin{theorem}\label{pazy}
Let $A$ generate the $C_0$-semigroup $(T(t))_{t \ge 0}$ on a Banach space $X$ and 
let $K : [0,\infty) \times X \to X$ be continuous and 
Lipschitz in $x$ on bounded sets: for every $T>0$ and every $r>0$ there is $C_{T,r}>0$ such that
\begin{equation*}
\| K(t, x)- K(t, y)\|\le C_{T, r}\|x-y\|
\end{equation*}
for $t \in [0,T]$ and $x, y \in B(0, r).$
 Then for every $x_0 \in X$ there exists a maximal existence time $T=T(x_0)\in (0,\infty]$ 
 such that the following holds.
\begin{itemize}
\item [(i)] There is a unique solution $x=x(\cdot, x_0)$
of \eqref{ca1}.
\item [(ii)] If $T<\infty,$ then $\lim_{t \to T}\|x(t, x_0) \| = \infty.$
\item [(iii)] For any $T^* \in (0, T)$ there exists a radius $\delta > 0$ such that
$T=T(y_0) > T^*$ for all $y_0 \in B(x_0, \delta).$ Moreover, the map
$B(x_0, \delta) \to C([0, T^*], X),$ $y_0 \to x(\cdot, y_0),$
is  continuous.
\end{itemize}
\end{theorem}
Gronwall's inequality and the above blow-up condition  (ii) imply the following global existence result.
\begin{corollary}\label{cor:lin-growth}
Besides the conditions of Theorem \ref{pazy}, assume that 
there exists $c\in L^1_{\mathrm{loc}}([0,\infty))$ such that
\begin{equation}\label{eq:lingrowth}
\|K(t, x)\|\le c(t)(1+\|x\|) \qquad \text{  for all } x \in X, \ t \ge0.
\end{equation}
Then for every $x_0\in X$ there is a unique solution $x(\cdot,x_0)$ of \eqref{ca1}
on $[0,\infty).$
\end{corollary}

See e.g.\ \cite[Corollary 6.2.3]{Pazy}, where the compactness of $(T(t))_{t \ge 0}$ is irrelevant under our assumptions.

In a similar way, the linear growth condition of the above corollary leads to  bounds on the solution
without assuming Lipschitz continuity of $K(t,\cdot)$. 

\begin{proposition}\label{prop:growth}
Let $A$ generate the $C_0$-semigroup $(T(t))_{t \ge 0}$ on a Banach space $X$ with
$\|T(t)\|\le M e^{\omega t}$ for $t \ge 0$ and some $\omega\ge0$, and  let $K : [0,\infty) \times X \to X$ be
continuous separately in both variables. 

\smallskip

a) Assume that $K$ satisfies \eqref{eq:lingrowth}.
Let $x:[0,T)\to X$ be a solution of  \eqref{ca1} and $t_0\in (0,T)$. Then $\|x(t)\|\le C(1+\|x_0\|)$
for $t\in[0,t_0]$ and a constant $C$ depending only on $t_0$, $M$, $\omega$ and $c(\cdot)$.

\smallskip

Let now \eqref{ca1} have a global solution $x=x(t,x_0)$ for all $x_0\in X_0$ and some subset $X_0\subseteq X$.

\smallskip

b) If \eqref{eq:lingrowth} holds,
then for every  bounded set $B \subset X_0$ 
 the set
$\{x(t, x_0):\, 0\le t\le t_0, \,\, x_0 \in B\}$
is bounded for each $t_0>0$.

\smallskip

c) Assume that $K$ is bounded on bounded sets and that
there exist $r_0 > 1$ and a non-increasing function
$c: [r_0, \infty) \to [0, \infty)$  such that 
$\|K(t,x)\| \le c(\|x\|)\|x\|$
for all $t\ge0$ and $x \in X$ with $\|x\|\ge r_0.$ 
 If $\gamma > \omega$ and $\lim_{r \to \infty} c(r) = 0,$ then there
exists a constant $M_{\gamma, x_0}$ such that $\|x(t,x_0)\|\le M_{\gamma, x_0} e^{\gamma t}$
for $t\ge0$ and $x_0\in X_0$ .
\end{proposition}
Assertion a) is a direct consequence of Gronwall's inequality (as in Corollary~\ref{cor:lin-growth}),
and part b) follows immediately from a). For the refinement in part~c), we refer to \cite{Webb}.
So, if $K$ is sublinear as $\|x\|\to\infty$, under the above assumptions the solutions of
\eqref{ca_problem} have the same growth bound as those for its linear part.

If  the mapping $x\mapsto K(t,x)$ is compact on $X,$ then one may drop
(local) Lipschitz type assumptions on $K$ at the price
of losing the uniqueness of  mild solutions to \eqref{ca_problem}.
(Roughly, one replaces Banach's fixed point theorem, guaranteeing the uniquneess of a fixed point,
with Schauder's fixed point theorem, where the uniquenesss is hardly available.)
However,  as we have already remarked above, for our purposes  mere existence suffices.
The next result from \cite{Sch83} is an example of existence theorems based on compactness properties
of the nonlinearity. We note that in this and related papers (such as \cite{Bothe98} mentioned below)
typically the concept of an integral solution is used. However, in our setting integral and mild solutions
coincide, see the Corollary to Proposition~3 in \cite{Sch83}.

\begin{theorem}\label{thm:cp1}
Let $A$ be the generator of a $C_0$-semigroup on a Banach space $X,$
and  $K: [0, \infty)\times X \to X$ be  jointly continuous and map bounded sets
in $[0, \infty)\times X$ to relatively compact sets in $X$.
Then for every $x_0 \in X$ there exists a maximal existence time $T=T(x_0)$ such that \eqref{ca1}
admits a solution $x$ on $[0,T).$ If  $T <\infty$, then $x$ is unbounded and
$\{K(t, x(t, x_0)): 0< t < T\}$ is not relatively compact.
Thus, in particular, if the range of  $K$ is relatively compact, then $T$ must be
infinite. 
\end{theorem}

\begin{remark} The statement is formulated in \cite[Theorem 2]{Sch83}
 for a semigroup of contractions. However, one can rescale the operators $T(t)$ and $K(t,\cdot)$
 as in the proof of Theorem~\ref{main}, thus reducing to a bounded semigroup $(T(t))_{t\ge0}$.
 Passing to the equivalent norm
$\|x\|_c:=\sup_{t\ge0} \|T(t)x\|$, the semigroup then becomes contractive, and we can apply  \cite[Theorem 2]{Sch83}.
\end{remark}

Combined with Proposition~\ref{prop:growth}, the above theorem yields the following result which we already used
after Theorem~\ref{main0}. For real Banach spaces, it was proved
in \cite{Bothe98} in a more general framework of
differential inclusions and with weaker compactness and regularity assumptions on $K$.
\begin{corollary}\label{thm:cp2}
Assume  the conditions of Theorem~\ref{thm:cp1} and \eqref{eq:lingrowth} hold. Then all solutions of \eqref{ca1}
exist on the whole of $[0,\infty).$
\end{corollary}

Finally, we will need a result yielding the existence of a unique propagator to \eqref{ca1}
in the linear setting, see e.g. \cite[Corollary VI.9.20]{EnNa00}.
\begin{proposition}\label{linear_ex}
Let $A$ be the generator of a $C_0$-semigroup $(T(t))_{t \ge 0}$ on Banach space $X,$
and let $K: [0,\infty)\to B(X)$ be strongly continuous.
Then there exists a unique evolution family $(U(t,\tau))_{t \ge \tau \ge 0}\subset B(X)$
such that
\begin{equation}\label{linear_var_con}
U(t,\tau)x= T(t-\tau)x + \int_{\tau}^{t}T(t-s)K(s)U(s,\tau)x\, ds, \qquad t \ge \tau.
\end{equation}
\end{proposition}

\section{Growth and instability for continuous time}\label{continuous}

Following a similar strategy as in Section \ref{discrete}, we now obtain lower bounds
for solutions of semilinear abstract differential equations with a linear part being a generator of 
a $C_0$-semigroup. In view of the failure of the spectral mapping theorem for $C_0$-semigroups, this task 
is more  demanding than the one treated in Section \ref{discrete}. However,
the ideas remain the same as for the systems with discrete time.

We start with proving auxiliary (and probably known) results on compactness.
Below we will assume that the mapping $K:[0,\infty)\times X\to X$ is
separately continuous and \emph{collectively compact}, i.e., the set $K\bigl( [0,t_0]\times B\bigr)$
is precompact in $X$ for each bounded set $B\subset X$ and each $t_0\ge0$.
For this notion in the linear case one may consult e.g.\ \cite{Anselone}.

A simple condition for collective compactness is provided by assuming continuity in $t$ uniformly
for $x$ in bounded subsets.
\begin{lemma}\label{lemmacom}
Let $K:[0,\infty)\times X\to X$ be separately continuous  such that $K(s,\cdot):X\to X$ 
is compact for all $s\ge 0$. Suppose that for each $t_0,\e>0$ and each bounded set $B\subset X$
there exists a  $\delta_\e>0$ 
such that
$$
\|K(s,x)-K(t,x)\|\le \e 
$$
for all $s,t\in[0,t_0]$ with $|s-t|\le\delta_\e$ and $x\in B$. Then  $K$ is collectively compact.
\end{lemma}
\begin{proof}
Let $t_0>0$ and $B\subset X$ be a bounded set. 
Take $\e>0$.
Let $\{t_1,\dots,t_n\}$ be a finite $\delta_{\e/2}$-net 
in the interval $[0,t_0]$.
For each $j=1,\dots,n$ the set $K(t_j)B$ is precompact and so is the set $C=\{K(t_j,b): j\in\{1,\dots,n\}, b\in B\}$.
Let $x_1,\dots,x_m$ be a finite $\frac {\e}{2}$-net for $C$. 

Let $0\le s\le t_0$ and $u\in B$. There exist $j\in\{1,\dots, n\}$ with $|s-t_j|\le\delta_{\e/2}$ 
and $i\in\{1,\dots, m\}$ with $\|K(t_j,u)-x_i\|<\e/2$. Using also the assumption, we obtain
$$
\|K(s,u)-x_i\|\le \|K(s,u)-K(t_j,u)\|+\|K(t_j,u)-x_i\|
\le \frac{\e}{2}+\frac{\e}{2}=\e. 
$$
So $\{x_1,\dots,x_m\}$ is a finite $\e$-net for the set $K\bigl([0,t_0]\times B\bigr)$. 
Since $\e>0$ was arbitrary, the set $K\bigl([0,t_0]\times B\bigr)$ is precompact.
\end{proof}

Before stating a second compactness criterion,
we collect our standing assumptions for semilinear evolution equations.

\begin{itemize}
\item[(A0)] Assume that $K:[0,\infty]\times X \to X$ is collectively compact
and  continuous separately in both variables, and let
$(T(t))_{t \ge 0}$ be a $C_0$-semigroup on a Banach space $X$ with generator $A.$
\item [(A1)] 
Let  (A0) hold, $X_0\subset X$ and
assume that for every $x _0 \in X_0$ there exists a unique global solution
$x(\cdot,x_0)\in C([0,\infty),X)$ of
\begin{equation}\label{xt}
x(t,x_0)=T(t)x_0+\int_0^t T(t-s)K(s, x(s,x_0)) \, ds, \qquad t \ge 0.
\end{equation}
So the map $x:[0,\infty)\times X_0\to X$ is well-defined, and we assume that it is continuous
separately in both variables.
\item [(A2)] Let (A0) hold, and $B\subset X_0$ be a bounded set. Assume that
for every $x_0\in B$ there exist a 
solution $x=x(\cdot,x_0)$ of \eqref{xt} such that the set
$\{x(t, x_0):\, 0\le t\le t_0, \,\, x_0 \in B\}$
is bounded for each $t_0>0$.
\end{itemize}

If (A0) holds and the solutions of \eqref{xt} are unique, then (A1) is satisfied for all $B \subset X_0$
provided that the map
 $x:[0,\infty)\times X_0\to X$ maps bounded into bounded sets. 

 In condition (A2), the map $(t,x_0)\mapsto x(t,x_0)$ can be defined for each bounded subset
 of $X_0,$ though it depends on this subset and the choice of corresponding solutions.
Note that  Corollary \ref{thm:cp2}
provides an example of $K$ satisfying (A2) for $X=X_0$ and all bounded sets $B \subset X_0.$
One can treat the non-unique case in a more systematic manner, for  instance using ``generalized semi-flows''.
See e.g.\ \cite{Ball},  where also nonlinear damped wave equations are treated.
 In this paper, we preferred however to avoid such concepts.

\begin{proposition}\label{prop1}
Let assumption (A2) hold, and let $t_0>0$ be fixed.
Let $C$ be the set whose elements are of the form
\begin{equation*}
\int_0^t T(t-s)K(s, x(s,x_0))\, ds
\end{equation*}
for all $x_0\in B$ and $t\in[0, t_0],$ where $x(\cdot, x_0)$ and $B$ are given by (A2).
Then $C$ is precompact in $X$.
\end{proposition}

\begin{proof}
Set $k=\sup\{\|T(t)\|:t\in[0, t_0]\}<\infty$ and $k'=\sup\{\|x(t)\|: t\in[0,t_0],\, x_0\in B\}<\infty$. 
Fix $\e>0$.

By assumption, the set $C_0=\bigcup_{0\le s\le t_0} \{K(s,y) : y\in X, \,\|y\|\le k'\}$ is precompact. 
Let $E$ be a finite $\frac{\e}{3kt_0}$-net in $C_0$. Since the set $\bigcup_{s\le t_0} T(s)(E)$ is precompact,  
Mazur's theorem yields the compactness of
$$
M:=[0,t_0]\cdot\overline{{\rm conv}}\,\Bigl(\bigcup_{s\le t_0} T(s)(E)\Bigr),
$$
where $\overline{{\rm conv}}$ stands for the closed convex hull.
Choose a finite $\frac{\e}{3}$-net $E'$ in $M$.
Let $0\le t\le t_0$,  $x_0\in B$, and $\e>0$. 
Then there exist $N\in\NN,$ a partition $0=t_0< t_1 < ... t_N=t$  of $[0,t],$ and
$s_{j}\in [t_j, t_{j+1}]$  for $0\le j \le N-1$
such that 
$$
\Bigr\|\int_0^t T(t-s)K(s,x(s,x_0))\, ds-S\Bigr\|<\frac\e3.
$$
where
$$
S= \sum_{j=1}^{N} T(t- s_{j}) K (s_{j}, x(s_{j},x_0)) (t_{j}-t_{j-1}).
$$

By the choice of $E$, for $j=1,\dots,N$ there are $e_j\in E$ such that
$$\|K(s_j,x(s_j, x_0))-e_j\Bigr\|<\frac{\e}{3kt_0}.
$$
Hence,
$$
\Bigl\|S-\sum_{j=1}^{N}T(t-s_j)(t_j-t_{j-1})e_j\Bigr\|<\frac{\e}{3},
$$
Since $\sum_{j=1}^NT(t-s_j)(t_{j}-t_{j-1})e_j$ belongs to $M$, we find
 $e'\in E'$ with 
$$ \Bigl\|\sum_{j=1}^{N}T(t-s_j)(t_j-t_{j-1})e_j-e'\Bigr\|<\frac{\e}{3}.
$$
It follows that
$$
\Bigl\|\int_0^t T(t-s)K(s,x(s, x_0))\, ds- e'\Bigr\|<\e,
$$
i.e., $E'$ is a finite $\e$-net for the set $C$. Since $\e>0$ was arbitrary, $C$ is precompact.
\end{proof}

The next immediate corollary of \eqref{xt} and Proposition \ref{prop1}
will play the role of Lemma \ref{lem1} in the continuous setting.

\begin{corollary}\label{lemcont}
Under the assumptions of Proposition \ref{prop1},
 there exists a compact set $C_{0}=C_0(t_0,B)\subset X$ such that
$$
x(t,x_0) \in T(t)x_0 +C_{0}
$$
for all $x_0\in B$ and $t \in [0,t_0]$.
\end{corollary}

The next statement is a counterpart of Theorem \ref{subeq}
providing an exponential growth bound on a residual set
at the expense of stronger assumptions on $K$ and a much smaller set where the bound holds.

\begin{theorem}\label{subseq}
 Assume that  (A1) holds with $X_0=X,$
and (A2) is satisfied for any bounded subset of $X.$
 Let $a:[0,\infty)\to\RR_+$ be a non-increasing function with $\lim_{t\to\infty} a(t) = 0$
and $L\subset\mathbb R_+$ be unbounded.
Then the set
$$
\bigl\{y\in X : \, \hbox{there are infinitely many} \,\, t_n\in L \,\,\hbox{with} \,\, \|x(t_n, y)\|\ge
a(t_n)\|T(t_n)\|_\mu\bigr\}
$$
is residual.
\end{theorem}

\begin{proof}
The proof follows that of Theorem \ref{subeq}, though we do not reduce it to this result.
Without loss of generality we assume that $\|T(t_0)\|_\mu\ne 0$ for all $t_0\ge 0$ 
(since otherwise $\|T(t)\|_\mu\le\|T(t_0)\|_\mu\cdot\|T(t-t_0)\|_\mu=0$ for all $t\ge t_0$.)
For $k\in\NN$ set 
$$
M_k=\{y\in X: \hbox{ there exists } t\in L\text{ with } t\ge k, \, \|x(t,y)\|>a(t)\|T(t)\|_\mu\}.
$$
Note that $M_k$ is open since $x$ is continuous in $y$. We show that $M_k$ is dense.

Let $y\in X$ and fix $\e > 0$. Choose $t\in L$ with $t\ge k$ and $a(t) < \e/4$. 
By Corollary~\ref{lemcont} there is a compact set $C\subset X$ such that$$x(t,y')\in T(t)y' + C$$
for all $y'\in X$ with $\|y'-y\|\le \e$. Since $C$ is compact, there exists a finite-dimensional subspace 
$F\subset X$ such that $\dist \{c, F\} < \frac{\e\|T(t)\|_\mu}{12}$ for all $c\in C$.
Lemma~\ref{lem3} provides a unit vector $u\in X$ with 
$\dist \{T(t)u, F\} >\frac{\|T(t)\|_\mu}{3}$. We compute
$$
\dist\! \{T(t)(y + \e u), F\} + \dist\! \{T(t)(y - \e u), F\} \ge
2\e\dist\! \{T(t)u, F\} \ge \tfrac{2\e}{3} \|T(t)\|_\mu.
$$
So $x_0:= y + \e u$ or $x_0 := y - \e u$ satisfy
$$\|x_0 - y\| \le\e \quad \text{and} \quad \dist \{T(t)x_0, F\}\ge \frac{\e}{3}\|T(t)\|_\mu.$$
It follows
$$
\|x(t,x_0)\|
\ge
\dist \{T(t)x_0, F\} - \frac{\e\|T(t)\|_\mu}{12}
\ge
\frac{\e\|T(t)\|_\mu}{4}\ge a(t)\|T(t)\|_\mu.
$$
Hence $x_0\in M_k$, and $M_k$ is dense since  $y\in X$ and $\e>0$ were arbitrary.
The Baire category theorem shows that
$\bigcap_{k=1}^\infty M_k$ is a
dense $G_\delta$ set, and thus residual.
\end{proof}

\begin{corollary}\label{corr2} 
Assume that all of the conditions of Theorem \ref{subseq} hold.
Let  $a:[0,\infty)\to\RR_+$ be a non-increasing function satisfying $\lim_{t\to\infty} a(t) = 0$.
Then there is a residual set $M \subset X$ such that for every $x_0 \in M$ there 
exist an unbounded sequence $(t_{n})_{n=1}^\infty$ with
\begin{equation*}
\|x(t_{n}, x_0)\|\ge a(t_n) e^{\omega_e(T)t_{n}}, \qquad n \in \mathbb N.
\end{equation*}
\end{corollary}

We next prove a continuous analogue of Theorem \ref{theorem1} which
 is one of the main results of this paper. It provides global and sharp exponential 
lower bounds for solutions of semilinear differential equations with unbounded linear part. Recall here
the notion of admissibility introduced in Section \ref{prelim}.

\begin{theorem}\label{main}
Let (A2) hold for $B=\overline{B}(y,r)\subset X_0.$
Setting  $\alpha=\limsup_{t \to 0}\|T(t)\|,$
let $a:[0,\infty)\to \RR_+$ be a non-increasing function satisfying $\lim_{t\to\infty}a(t)=0$, 
$t_0\ge 0$ with $a(t_0)<\frac{r}{2\al}$.
Assume that $\omega \in \mathbb R$ is admissible. Then there exists $z \in \overline{B}(y,r)$ 
and a solution $x(\cdot, z)$ to \eqref{xt} with $x_0=z$ such that
$$
\|x(t,z)\|\ge a(t) e^{\omega t}, \qquad t\ge t_0.
$$
\end{theorem}
\begin{proof}
Without loss of generality we may assume that $\omega=0$. If not, then consider the semigroup
$(e^{-\omega t}T(t))_{t \ge 0}$  and compact perturbations $e^{-\omega s}K(e^{\omega s} \cdot)$.
Define a new norm $\|\cdot\|_1$ on $X$ by
$$
\|x\|_1=\sup\{|\langle x,x^* \rangle|:\, x^*\in D(A^*),\, \|x^*\|=1\}.
$$
By \eqref{renorm}, the norm $\|\cdot\|_1$ is equivalent to $\|\cdot\|$ with
$$
\|x\|_1\le\|x\|\le \al\|x\|_1
$$
for all $x\in X$.
Let a Banach space $X_1$ be $X$ equipped with $\|\cdot\|_1.$

Fix $c_1$ such that $2a(t_0)<c_1<r$ and $r-c_1<\frac12$.  Pick an increasing sequence $(t_k)_{k=0}^\infty$
with $a(t_k)<\frac{r-c_1}{\alpha 2^{k+2}}$ for all $k \in \mathbb N\cup \{0\}.$ Set 
\[ c_k=\frac{r-c_1}{2^{k-1}}, \qquad k\ge 1.
\]
For each $k \in \mathbb N$ choose  $\e_k>0$ such that $\e_{k+1}<\e_k$ and
$$
\frac{(1-\e_k)^2}{2\alpha}c_k-2\e_k\ge a(t_{k-1})
$$
for all $k\ge 1$.

The required element $z$ will be constructed as the sum of a series of appropriate approximate eigenvectors $x_k$
for $(T(t))_{t \ge 0}$ and $e^{\mu_k t}$. We set $F_0=\{0\},$  $M_0=X$, and $x_0=y$. Inductively we construct
unit vectors
$(x_k)_{k=1}^\infty \subset X_1$, finite-dimensional subspaces $F_1\subset F_2\subset\cdots$, and closed
finite-codimensional subspaces $M_1\supset M_2\supset\cdots$ with $x_k\in M_k\cap F_{k+1}$ of $X_1$
which satisfy \eqref{eq:ind1}--\eqref{eq:ind4} below.

Let $k\ge 1$ and suppose that the vectors $x_0,\dots,x_{k-1}$ and subspaces $F_0,\dots,F_{k-1}$ and 
$M_0,\dots,M_{k-1}$ have already been constructed.
Consider the family of 
solutions $\{x(t, u) : u\in B, t \ge 0\}$ to \eqref{xt} given by (A2),
and fix the notation $x(t, \cdot)$ for this family.
If $B_1$ is the image of $B$ under the identity embedding of $X$ into $X_1$, then
Corollary~\ref{lemcont} yields a compact set $C_k\subset X$ such that $x(t,u)\in T(t)u+C_k$ 
for all $t\le t_k$ and $u \in B_1.$
There thus exists a finite-dimensional subspace $F_k\supset F_{k-1}$ 
such that  $x_{k-1}\in F_k$ if $k>1$ and
\begin{equation}\label{eq:ind1}
\dist\Bigl\{T(t)\Bigl(y+\sum_{j=1}^{k-1} \al^{-1}c_jx_j\Bigr),F_k\Bigr\}\le\frac{\e_k}{2}
\end{equation}
for all $t\le t_k$ 
and
\begin{equation}\label{eq:ind2}
\dist\{d,F_k\}\le\e_k/2
\end{equation}
for all $d\in C_k$, where $\dist$ stands for the distance in the new norm $\|\cdot\|_1$.
Lemma~\ref{geomlemma} yields a closed subspace $M_k\subset M_{k-1}$ of finite codimension such that
\begin{equation}\label{eq:ind3}
\|f+m\|_1\ge (1-\e_k)\max\Bigl\{\|f\|_1,\frac{\|m\|_1}{2}\Bigr\}
\end{equation}
for all $f\in F_k$ and $m\in M_k$, where
\[
M_k=\bigcap_{1\le j \le k} {\rm Ker}\, y_j^*
\]
for some $y_j^* \in D(A^*)$ for $1 \le j \le k.$
Since $0$ is admissible, we can choose $x_k\in M_k$ with $\|x_k\|_1=1$ and $\mu_k\in\CC$ with $\Re\mu_k=0$ 
such that
\begin{equation}\label{eq:ind4}
\|T(t)x_k -e^{\mu_k t}x_k \|_1 \le\e_k
\end{equation}
for all $t\le t_k$.

Suppose that the vectors $x_k$, $k\in\NN,$ have been constructed as above. Set
$$z=y+\sum_{j=1}^\infty \al^{-1}c_j x_j.$$
We show that $z$ meets the requirements in the initial Banach space $(X, \|\cdot\|).$
First, $z$ belongs to $\overline{B}(y,r)$ since
$$\|z -y\|\le \al\|z-y\|_1\le \sum_{j=1}^\infty c_j=  c_1+\sum_{j=2}^\infty \frac{r-c_1}{2^{j-1}}=r.$$
Let $x(\cdot, z)$ be a 
solution of \eqref{xt} given by (A2).
Fix $k\ge 1$ and consider $t \in [t_{k-1},t_k]$. Properties \eqref{eq:ind2} and \eqref{eq:ind1} yield
\begin{align*}
\|x(t, z)\|_1&\ge\dist\{T(t)z,C_k\}\ge\dist\{T(t)z,F_k\}-\e_k/2\\
&\ge\dist\Bigl\{\sum_{j=k}^\infty \al^{-1}c_j T(t)x_j,F_k\Bigr\}-\e_k.
\end{align*}
By means of \eqref{eq:ind4} we estimate
\begin{align*}
\|x(t, z)\|_1&\ge \dist\Bigl\{\sum_{j=k}^\infty \al^{-1}c_j e^{\mu_j t}x_j,F_k\Bigr\}
          -\sum_{j=k}^\infty \al^{-1}c_j\e_j -\e_k\\
&\ge \dist\Bigl\{\sum_{j=k}^\infty \al^{-1}c_je^{\mu_j t}x_j,F_k\Bigr\}-2\e_k.
\end{align*}
Taking into account \eqref{eq:ind3}  and that   $x_j\in M_j\subset M_k$ for all $j\ge k$, 
we infer 
$$
\|x(t, z)\|_1\ge
\frac{1-\e_k}{2}\Bigl\|\sum_{j=k}^\infty \al^{-1}c_j e^{\mu_j t}x_j\Bigr\|_{1}-2\e_k.
$$
Since $x_k\in F_{k+1}$ and $x_j\in M_j\subset M_{k+1}$ for $j\ge k+1$, \eqref{eq:ind3} also implies
\begin{align*}
\|x(t, z)\|_1&\ge \frac{(1-\e_k)(1-\e_{k+1})}{2}\|\al^{-1}c_k e^{\mu_k t}x_k\|_{1} - 2\e_k\\
&\ge \frac{(1-\e_k)^2}{2}\cdot \al^{-1}c_k -2\e_k\ge a(t_{k-1})\ge a(t).
\end{align*}
Hence
$$
\|x(t, z)\|\ge \|x(t, z)\|_1\ge a(t)
$$
for all $t\ge t_0$, as required.
\end{proof}

\begin{corollary}\label{main2}
(i) Let the assumptions of Theorem \ref{main} hold with
$r> 2a(0)\al$.
    Then there exists $y\in X_0$ with $\|x(t, y)\|\ge a(t)e^{\omega t}$ for all $t\ge 0$.

    (ii) Assume that assumption (A2)
is satisfied for all bounded subsets $B$ of $X_0.$
Then there exists a set $S \subset {\rm Int}\,X_0$ dense in ${\rm Int}\, X_0$ such that
  $\|x(t, y)\|\ge a(t)e^{\omega t}$ for all $y \in S$ and $t$ sufficiently large.
\end{corollary}
\begin{proof}
To deduce (i) from Theorem \ref{main}, it suffices to set $t_0=0,$ and to let $y$ be the center of the ball in $X_0$ of radius
greater than $2a(0)\al$.

 To prove (ii), we may assume that ${\rm Int} X_0\neq \emptyset.$ Let $y\in {\rm Int}\,X_0$,
 $\e=\dist\{y,\partial X_0\}$, and $\e_0 \in (0,\e).$ Find $t_0\ge0$ such that
 $a(t_0)<\frac{\e_0}{2\alpha}$. By Theorem \ref{main}, there exists $x_0 \in X$ such that $\|x_0-y\|\le\e_0$ and 
 $\|x(t,x_0)\|\ge a(t)e^{\omega t}$ for all $t\ge t_0$. It remains to note that the choice of $\e_0\in (0,\e)$ is arbitrary.
\end{proof}

\begin{remark}
If $X$ is a Hilbert space, as in Section \ref{discrete} one can  obtain a better estimate.
Let $y\in X$, $\{u:\|u-y\|\le r\}\subset X_0$,  $t_0\ge 0$ and $a(t_0)<r.$ 
Then there exists $x_0 \in X$ with $\|x_0-y\|\le r$ and $\|x(t, x_0)\|\ge a(t)e^{\omega t}$ for  $t\ge t_0$.
\end{remark}

The next corollary is one of the main results of the paper.
It is a direct consequence of Theorem \ref{proplower} and Theorem \ref{main}. 
\begin{corollary}\label{main1}
Under the conditions of Theorem \ref{main}, there exists $x_0 \in \overline{B}(y,r)$ and $t_0\ge0$ such that
\begin{equation}\label{bound1}
\|x(t,x_0)\|\ge a(t) e^{s_R(A) t}, \qquad t\ge t_0.
\end{equation}
If $X_0=X$ and the assumption (A2) is true
for all bounded subsets $B$ of $X,$ then the set of $x_0$ satisfying \eqref{bound1} is dense in $X.$
\end{corollary}
If in \eqref{bound1} the operator $A$ is bounded and $s_e(A)>0,$ then under appropriate local assumptions
on $K$ (involving compactness)  the local instability of zero solution to \eqref{bound1} was shown in \cite{Slyus2}.

Our results on lower bounds are also new in the framework of linear equations \eqref{ca1},
where $K: [0,\infty)\to B(X)$ is strongly continuous and each operator $K(t)$ is compact.
In this case, $K$ does not have to be collectively compact and so (A0) may be violated.
The result is new even for norm-continuous $K$, where (A0) holds by Lemma~\ref{lemmacom}.)
However,  in the proof of
Theorem~\ref{main} we need collective compactness only to apply Corollary~\ref{lemcont},
and this result  is a direct consequence of Proposition~\ref{linear_ex}
and \cite[Theorem C.7]{EnNa00} in the linear case. Hence, Theorem \ref{proplower} and  Corollary \ref{main2}
yield the next estimate.
\begin{corollary}\label{linear}
Let $A$ be the generator of a $C_0$-semigroup on Banach space $X$
and  $K: [0,\infty)\to B(X)$ be a strongly continuous function
such that $K(t)$ is a compact operator for each $t \ge 0.$
Let $(U(t,\tau))_{t \ge \tau \ge 0}$ be  the evolution family given by  Proposition \ref{linear_ex}.
Assume that $a:[0,\infty)\to [0,\infty)$ is a decreasing function satisfying $\lim_{t \to \infty}a(t)=0$.
Then there exists  a dense set of vectors $x$ such that
\[
\|U(t,0)x\|\ge a(t)e^{s_R(A)t}, \qquad t \ge t_0=t_0(x).
\]
\end{corollary}

We finish this section with a discussion of our results.
First, note that one may treat $s_R(A)$ ``up to compact perturbations'': In the above result one may
consider $s_R(A+S)$ for a compact perturbation $S\in B(X)$ of $A$ and substract $S$ from the nonlinearity
$K$ in \eqref{xt}, without changing the assumptions on $K.$

\begin{remark}\label{relative}
Theorem \ref{main} and Corollaries \ref{main2} and \ref{main1}  in general do not hold for $K$ being merely relatively
compact with respect to $A,$ i.e.,  such that $K:D(A)\to X$ is compact, where $D(A)$ is equipped with the graph norm.
For instance, consider the setting of linear damped wave equations
with the operator $A_-$  given by \eqref{damped_equa} below. Here we have
\[
A_-=D+K, \qquad D:=\begin{pmatrix} 0 & I\\ \Delta & 0 \end{pmatrix}, \qquad
K:=\begin{pmatrix} 0 & 0\\ 0 & -b \end{pmatrix},
\] 
where $D$ generates a unitary $C_0$-group on an appropriate Hilbert space $X$, $s_R(D)=0,$ $K$ is relatively compact
with respect to $D$. However, in view of
\cite{BarLeTa}, the operator $A_-$  generates an exponentially stable $C_0$-semigroup on $X$
if  $b \in C^\infty(\cM)$ and $b$ satisfies the so-called  geometric control condition. (See also e.g.\
\cite{Anan_L}, \cite{JL}, and \cite{Lebeau} for a relevant discussion
and some examples.) Thus, in this case, we have $\omega_0(A_-)<0,$ which excludes any results of the form 
of Theorem  \ref{main} and its corollaries with $A=D$ and $K$ as above. If $K$ is compact, then the situation changes as we show in the next section. 
\end{remark}

Finally, we address the optimality of results stated in Theorem~\ref{main} and Corollary~\ref{main1}.

\begin{remark}\label{rem:opt}
Let $(T(t))_{t \ge 0}$ be the left shift $C_0$-semigroup on $L^2(\mathbb R_+)$ defined by
$(T(t)f)(s)=f(s+t), t \ge 0,$ for $f \in L^2(\mathbb R_+).$
If $A$ is the generator of $(T(t))_{t \ge 0},$ then 
$\sigma(A)=\{\lambda \in \mathbb C: {\rm Re}\, \lambda \le 0\},$ and $s_e(A)=s_0(A)=s_R(A)=0.$  On the other hand, clearly $T(t) \to 0$ as
$t \to \infty$ strongly.
Thus  Theorem \ref{main} and Corollary \ref{main1} cannot be improved by removing the function $a$ from their formulations,
even for linear equations \eqref{ca1} with $K=0.$ For other concrete examples of such semigroups one may consider the the multiplication semigroups $(T(t)f)(z)=e^{-tz}f(z)$, $t \ge 0,$
on appropriate spaces $L^2(\Omega)$ with  $\Omega$ from the closed left half-plane.
\end{remark}

\section{Backward damped wave equations and other applications}\label{sec:exa}
In this section, we illustrate our abstract results by applying them to
the study of backward damped wave equations or, equivalently, ``excited'' wave equations,
subject to nonlinear forcing terms of the form $-f(t,\cdot,u)$. (The minus is chosen
to simplify  the sign condition in \eqref{eq:sign}.)
In this  basic PDE setting a positive abscissa
$s_R(A)>0$ and compact nonlinear perturbations $K(t,\cdot)$ occur naturally,
including situations where positivity of $s_R$ stems from the resolvent growth.
In contrast to earlier work, as in e.g.\ in \cite{Sola}, we allow for time-dependent and merely continuous $f$
so that we cannot expect uniqueness in general.  As we focus on asymptotic properties of solutions,
our assumptions on nonlinearities are, of course, not best possible, and serve first of all
to create the right framework to the study of lower bounds for solutions in spectral terms.

We first look at the excited wave equation
\begin{align}
\partial_{tt} u(t,x) -\Delta u(t,x) -b(x)\partial_t u(t,x) &= -f(t,x,u(t,x)), \,\, x\in \mathcal M, \ \ t\in \mathbb R,\label{eq:excited} \\
u(0,x)=u_0(x),  \quad \partial_tu(t,0)&=u_1(x), \qquad x\in \mathcal M,\label{eq:excited2}
\end{align}
on a $d$-dimensional, compact, smooth and connected Riemannian manifold $\mathcal M$ without or with boundary
$\partial \mathcal M.$
Here $\Delta$ is the Laplace--Beltrami operator on $\mathcal M,$ depending in general on a metric on $\mathcal M.$
We do not indicate this dependence since it will not be relevant. If $\partial \mathcal M \neq \emptyset$,
we additionally impose Dirichlet boundary conditions in \eqref{eq:excited}.

In order to apply the results from the previous section, we have to consider complex-valued $u$ in
\eqref{eq:excited} and related equations. Since we do not want to restrict ourselves to holomorphic maps
$\zeta\mapsto f(t,x,\zeta)$, we identify $\CC$ with $\RR^2$ equipped with the Euclidean scalar product
$\zeta_1\cdot \zeta_2=\xi_1\xi_2+\eta_1\eta_2=\Re(\zeta_1\overline{\zeta_2})$ where
$\zeta_j=\xi_j+i\eta_j=(\xi_j,\eta_j)$. Differentiability is then understood in the real sense and derivatives
are only $\RR$-linear. Fortunately, this does not affect the basic rules from calculus that we use here.
Throughout we assume that
\begin{equation}\label{eq:f}
\begin{split}
&b\in L^\infty(\cM), \qquad b \ge 0, \qquad b>0  \,\, \text{\ on an open subset of}\,\, \cM,\\
&f:\RR\times \cM\times\CC\to\CC \text{ \ is continuous in
 the $3$d variable and measurable},\\
&|f(t,x,\zeta)|\le \kappa(t)(1+|\zeta|^\alpha), \qquad  |f(t,x,\zeta)-f(s,x,\zeta)|\le \omega(|t-s|) (1+|\zeta|^\alpha),
\end{split}\end{equation}
for some $0\le\alpha< d/(d-2)_+$ and all $t,s\in\RR$, $\zeta\in \CC$,  $x\in \cM$,
where $\kappa: \mathbb R \to [0,\infty)$ is locally bounded and
$\omega:[0,\infty)\to [0,\infty)$ satisfies $\omega(\tau)\to 0$ as $\tau\to0.$
(Here $y_+=\max\{y,0\}$ for $y\in\RR$ and $\frac{d}{0}:=\infty$.)
The assumption $b\ge0$ fits to the interpretation of \eqref{eq:excited}
as excited  wave equation.
It is crucial to observe that \eqref{eq:f} implies the continuity of $f$ in $(t,\zeta).$

To formulate \eqref{eq:excited},\eqref{eq:excited2} as an evolution equation of the form \eqref{solaeq},
we set $V=H^1(\cM)$ if $\cM$ has no boundary  and $V=H^1_0(\cM)$ otherwise.
Recall that $V$ is compactly embedded into $L^{2\alpha}(\cM)$, see Theorem~2.34 of \cite{Aubin}.
On the state space $X=V\times L^2(\cM)$ we introduce the operator matrix
\begin{equation}\label{excited_equa}
 A_+= \begin{pmatrix} 0 & I\\ \Delta & b \end{pmatrix}, \qquad D(A_+)= (H^2(\cM)\cap V)\times V.
\end{equation}
It is well known that $A_+$ generates a $C_0$-group $(T_+(t))_{t\in\RR}$ on $X$, i.e.,
$\pm A_+$ are generators of $C_0$-semi\-groups. We write elements of $X$ as $w=(u,v)$.
The forcing term is expressed by $K(t,w)= (0,-f(t,\cdot, u))$.
As we will prove in Proposition~\ref{lem:K1}  below,
$K:\RR\times X\to X$ is jointly continuous and collectively compact.

We study \eqref{eq:excited} in forward time.
Using standard properties of the wave equation with the (continuous) right-hand side
$t\mapsto h(t):=-f(t,\cdot,u(t))$, it can be checked that a mild solution $w$ to
\begin{align}\label{eq:excited1}
w'(t)= A_+w(t) +K(t,w(t)), \quad t\ge0,\quad w(0)=w_0:=(u_0,u_1)\in X,
\end{align}
on a time interval $J=[0,T)$ is of the form $w=(u,\partial_t u)$ for a function $u$
in $C^2(J,H^{-1}(\cM))\cap C^1(J, L^2(\cM))\cap C(J,V)$ which solves \eqref{eq:excited}.
(Cf.\ \cite[Lemma~6.2.1]{Cazenave}).

For the damped case, we replace $+b$ by $-b$ obtaining the operator
 \begin{equation}\label{damped_equa}
A_-= \begin{pmatrix} 0 & I\\ \Delta & -b \end{pmatrix}, \qquad D(A_-)=D(A_+). 
\end{equation}
It generates a $C_0$-group $(T_-(t))_{t\in\RR}$,  which is contractive in forward time since $b\ge0$,
and it corresponds to the damped wave equation
\begin{align}
\partial_{tt} u(t,x) -\Delta u(t,x) +b(x)\partial_t u(t,x) &= -f(t,x,u(t,x)), \,\, x\in \cM, \ \ t\in \mathbb R, \label{eq:damped}\\
u(0,x)=u_0(x), \quad \partial_tu(t,0)&=u_1(x), \quad x\in \cM.\label{eq:damped2}
\end{align}
The above remarks on the solution also apply here.
We study \eqref{eq:damped}
 \emph{backwards} in time, i.e., for $t \le 0.$
To bring it to standard forward form,
one looks at $\tilde{u}(t)=u(-t)$ for $t \ge 0$
satisfying
\begin{align}
\partial_{tt} \tilde{u}(t,x) -\Delta \tilde{u}(t,x) -b(x)\partial_t \tilde{u}(t,x)
           &= -f(-t,x,\tilde{u}(t,x)), \, \, x\in \cM, \, t\ge 0, \label{eq:backward}\\
\tilde{u}(0,x)=u_0(x),  \quad \partial_t\tilde{u}(t,0)&=-u_1(x), \quad x\in \cM.\label{eq:backward2}
\end{align}
This system coincides with \eqref{eq:excited},\eqref{eq:excited2}  except for the additional minus in $f$
and before $u_1$. We drop the
tilde. To rewrite this problem as \eqref{solaeq}, we use the operators $B=-A_-$ and $w\mapsto -K(-t,w)=(0,f(-t,\cdot,u))$.
Then \eqref{eq:backward},\eqref{eq:backward2} can be reformulated as the first-order problem
\begin{align}\label{eq:backward1}
w'(t)= Bw(t)- K(-t,w(t)),\qquad t\ge 0, \qquad w(0)=(u_0, -u_1).
\end{align}
Moreover, if $J: X \to X$ is defined by $J(u,v)=(u,-v),$ then $A_+=J(-A_-)J^{-1}$
(as unbounded operators, see \cite[p. 59]{EnNa00}), so that $\sigma(A_+)=\sigma(-A_-)$
and $R(\lambda, A_+)=J R(\lambda, -A_-)J^{-1}$ for all $\lambda \in \rho(A_+).$
Thus,  
\begin{equation}\label{sr_adjoint_wave}
s_R(-A_-)=s_R(A_{+}).
\end{equation}
To clarify the relations between $A_+$ and $-A_-,$ note that by e.g.\ \cite[Lemma 1, p.75]{Yak}
we have  $(A_+)^*=-A_-,$ and thus one obtains \eqref{sr_adjoint_wave}
once again. 

It is well-known that $\sigma(A_-)\subset \{\lambda: -\|b\|_\infty\le {\rm Re}\, \lambda \le 0\}$
since $b\ge0$, that $\sigma(A_-)$ is invariant under conjugation,  and that 
it consists of a discrete set of eigenvalues since $A_-$ has compact resolvent.
Moreover if $b\ge0$ is non-zero, we have  $\sigma(A_-)\cap i\mathbb R=\emptyset$ if $\partial \mathcal M \neq \emptyset$
and $\sigma(A_-)\cap i\mathbb R=\{0\}$ (with constants as  eigenfunctions)
if $\mathcal M$ has no boundary. (See e.g. \cite[p.74]{Lebeau}, where the smoothness of $b$ assumed there
does not play a role in these statements, or \cite[Section 4]{Anan_L}.)

In the latter case, let $P_0$ be the Riesz projection corresponding to $0$
and equip $X_0=(I-P_0)X$ with the inner product norm
\[
\|(u_0, u_1)\|_{X_0}=
\|(-\Delta)^{1/2}u_0\|_{L^2} +\|u_1\|_{L^2}, \qquad (u_0,u_1) \in X_0.
\]  
Then $\dot T_{-}(t):=T_-(t) \upharpoonright_{X_0}, t \in \mathbb R,$ is a $C_0$-group on the Hilbert space $X_0$ generated by
$\dot A_{-}:=A_- \upharpoonright_{X_0},$ which is contractive for $t \ge 0.$ Moreover, by e.g.\ \cite[Section 4]{Anan_L},
we have $\sigma(\dot A_{-})=\sigma(A)\setminus \{0\} \subset \{\lambda: {\rm Re}\, \lambda \le 0\}$ and
there are $c_1,c_2>0$ with
\[
c_1\|R(\lambda, A_-)\|_{B(X)} \le \|R(\lambda, \dot A_{-})\|_{B(X_0)} \le c_2\|R(\lambda, A_-)\|_{B(X)}
\] 
for  $\lambda \in \rho(A_-)$  with $|\lambda| \ge\epsilon_0$ for an appropriate $\epsilon_0>0.$
This construction allows one, in particular, to study the energy asymptotics for \eqref{eq:damped} in a unified manner,
see e.g.\ \cite{Anan_L} (assumption \cite[(2-6)]{Anan_L} holds due to \cite[p.74]{Lebeau}), and also
\cite{Chill} and \cite{BBT}.
We will also use $\dot A_{-}$ in the sequel to setudy  the resolvent of $A_-$, see Example \ref{dyatlov}.

Before we consider spectrum and resolvent of $A_\pm$
in specific cases, we first establish the required properties of $K$,
\begin{proposition}\label{lem:K1}
Let $f$ and $b$ satisfy \eqref{eq:f}. Then the map $K:\RR\times X\to X$ defined
above is jointly continuous and collectively compact. If also
\begin{equation}\label{eq:lin}
|f(t,x,\zeta) |\le \kappa(t)(1+|\zeta|)
\end{equation}
for $\kappa$ from \eqref{eq:f} and all $(t,x,\zeta)\in\RR\times \cM\times \CC$,
 then  $\|K(t,w)\|\le c\kappa(t)(1+\|w\|)$ for all $w \in X$ and some $c>0$.
\end{proposition}
\begin{proof}
The last claim is clear. For the first, let $(t_n)_{n=1}^{\infty}\subset \mathbb R$
and $(w_n)_{n=1}^\infty \subset X$ be such that $t_n\to t$ in $\RR$ and $w_n=(u_n,v_n)\to w=(u,v)$ in $X$ as $n \to \infty.$
Since $V$ is compactly embedded in $L^{2\alpha}(\cM)$, there is a subsequence $(n_k)$ and a map
$g\in L^{2\alpha}(\cM)$ such that  $u_{n_k}\to u$ in $L^{2\alpha}(\cM)$ and pointwise a.e.\ as $k\to\infty$ and
$|u_{n_k}|\le g$ a.e.\ for all $k$. As noted above, $f$ is jointly continuous in $(t,\zeta),$ and
since $\kappa$  is locally bounded
we have $m:=\sup_n \kappa(t_n)<\infty$.
Hence, $f(t_{n_k},\cdot, u_{n_k})$ tends pointwise a.e.\ to $f(t,\cdot,u)$ as $k \to \infty$ and
$|f(t_{n_k},\cdot,u_{n_k})|\le m(1+|g|^\alpha)\in L^2(\cM)$ by \eqref{eq:f}.
It follows $f(t_{n_k},\cdot,u_{n_k})\to f(t,\cdot,u)$in $L^2(\cM)$ as $k \to \infty$, and so $K$ is continuous.

To prove compactness, take $t\in\RR$ and a bounded sequence $(w_n)_{n=1}^\infty$ in $X$. Again, there is a
subsequence $(n_k)$ and maps
$g\in  L^{2\alpha}(\cM)$ and $u\in V$ such that  $u_{n_k}\to u$ in $L^{2\alpha}(\cM)$ and pointwise a.e.\ as $k\to\infty$
and $|u_{n_k}|\le g$ a.e.\ for all $k$. As above, we infer that $f(t,\cdot,u_{n_k})$ tends to $f(t,\cdot,u)$
in $L^2(\cM)$, and thus $K(t,\cdot):X\to X$ is compact. To use Lemma~\ref{lemmacom}, let $t_0,r>0$, $t,s\in[-t_0,t_0]$
and  $\|w\|\le r$. Due to Sobolev's embedding, we then obtain $\|u\|_{L^{2\alpha}}\le Cr$
for a constant $C>0$ independent of $r,$ and \eqref{eq:f}
yields
\begin{align*}
 \|f(t,\cdot,u)-f(s,\cdot,u)\|_{L^2(\cM)}\le& \omega(|t-s|) \|1+ |u|^\alpha\|_{L^2(\cM)}\\
        \le&  \omega(|t-s|) (\operatorname{vol}(\cM)^\frac12 + C^\alpha r^\alpha).
\end{align*}
As $\omega(\tau)\to0$ as $\tau\to0$, the map $K$ is collectively compact by Lemma~\ref{lemmacom}.
\end{proof}

Due to this result, Corollary~\ref{thm:cp2} and Proposition~\ref{prop:growth}, if
$\alpha\le 1$ in \eqref{eq:f} then the operators $\pm K(\pm t,\cdot)$ fit to Theorem~\ref{main} and its corollaries.

If $1<\alpha<d/(d-2)_+$, we need a sign condition and some extra time regularity of $f$
to show global existence by means of a standard energy estimate, cf.\ e.g.\ \cite[Chapter~6]{Cazenave}.
We define the potential of $\zeta\mapsto f(t,x,\zeta)$ by the (real) line integral
\begin{equation}\label{eq:ph}
\ph(t,x,\zeta)=\int_0^1 f(t,x,\tau \zeta)\cdot \zeta d\tau
\end{equation}
for $(t,x,\zeta)\in \RR\times \cM\times \CC$. Notice that
\begin{equation}\label{est:ph}
|\ph(t,x,\zeta)|\le c\kappa(t)(1+|\zeta|^{\alpha+1}), \qquad (t,x,\zeta)\in \RR\times \cM\times \CC,
\end{equation}
for a constant $c>0$, by \eqref{eq:f}. In the next result it is assumed  that $\ph$ is differentiable in $\zeta$ with
$\nabla_\zeta \ph=f$. We first discuss this assumption.
\begin{remark}
If $f$ is $C^1$ in $\zeta=(\xi,\eta)$, then the same is true for $\varphi$.
We then have $\nabla_\zeta\varphi= f$ if and only if
$\partial_\eta f_1=\partial_\xi f_2$.
However,  $\nabla_\zeta \varphi$ also exists and is equal to $f$ for the standard example
$f(t,x,\zeta)=\psi(t,x,|\zeta|^2)\zeta$, where $\psi:\RR\times \cM\times \RR\to\RR$  is continuous
in the third variable. Here we do \emph{not} need differentibility of $f$ in $\zeta$ and obtain
\[\ph(t,x,\zeta)=\frac12\int_0^{|\zeta|^2}\psi(t,x,r)dr.\]
\end{remark}

We note that in the next proposition we do no use the results discussed in Section~\ref{sec:wp},
besides a variant of the basic Theorem~\ref{pazy}.
Instead we  construct the solutions as weak* limits of (subsequences of) solutions
to regularized problems. This standard method is based on uniform bounds for the energy and the compactness of $K$.
(See e.g.\ \cite[Section~4.4]{Temam} for a similar approach using a Galerkin approximation.)
\begin{proposition}\label{lem:K2}
Let $b$ and  $f$ satisfy \eqref{eq:f} with $1\le \alpha <d/(d-2)_+$ and $\ph$ be given by \eqref{eq:ph}. Assume that
$\ph$ is differentiable in the third variable,
\begin{equation}\label{eq:sign}
\ph\ge0, \qquad \nabla_\zeta \ph=f, \qquad |f(t,x,\zeta)-f(s,x,\zeta)|\le  \ell(t_0)|t-s|(1+|\zeta|)
\end{equation}
for $t_0>0$, $t,s\in[-t_0,t_0]$, $\zeta\in\CC$, $x\in \cM$, and a locally bounded map $\ell:[0,\infty)\to [0,\infty)$.
Then \eqref{eq:excited1} and \eqref{eq:backward1} have global
solutions $w$ such that $\|w(t)\|\le c(t_0,r)$ for $0\le t\le t_0$ and $\|(u_0,u_1)\|\le r$ and every $r,t_0>0$.
\end{proposition}
\begin{proof}
We only treat \eqref{eq:excited1} since \eqref{eq:backward1} is completely analogous.
For the energy estimate we need classical solutions of \eqref{eq:excited} so that we
approximate $f$ and $w_0=(u_0,u_1)\in X$ by more regular functions. Let $t\in\RR, $ $x\in \cM$, $\zeta \in\CC$, $t_0>0$, and
$n \in \mathbb N.$ Take standard mollifiers $\rho_n:\RR^2\to[0,\infty)$ with support in $\overline{B}(0,\frac1n)$ and
functions $\chi_n\in C^1(\RR^2,\RR^2)$ with range in $B(0,n+1)$, $\chi_n(\zeta)=\zeta$ for $|\zeta|\le n$,
$|\chi'_n(\zeta)|\le c$, and $|\chi_n(\zeta)|\le c|\zeta|$ for a constant $c\ge1$. We set
\begin{align*}
\tilde{\ph}_n(t,x,\zeta)&=\ph(t,x,\chi_n(\zeta)), \qquad \ph_n(t,x,\cdot)=\rho_n\ast  \tilde{\ph}_n(t,x,\cdot),\\
\tilde{f}_n(t,x,\zeta)&=\chi_n'(\zeta)^\top f(t,x,\chi_n(\zeta)), \qquad
f_n(t,x,\cdot)=\rho_n\ast  \tilde{f}_n(t,x,\cdot).
\end{align*}
It is straightforward to check that $\nabla_\zeta \ph_n=f_n$ and $\ph_n\ge0$ for all $n$. Moreover, for each $n \in \mathbb N,$
the maps $\tilde{f}_n,$ $f_n,$ $\partial_\zeta f_n$, $\partial_t f_n$ are bounded on $[-t_0,t_0]\times \cM\times\CC$
 (in general not uniformly in $n$) and there exist $c_{n,k}=c_{n,k}(t_0) \ge 0$ such that
 $|\partial_\zeta^k\ph_n(t,x,\zeta)|\le c_{n,k}(1+|\zeta|^{1+\alpha})$
for $(t,x,\zeta) \in [-t_0,t_0]\times \cM\times\CC$ and $k\in\NN_0$. Finally, $f_n$ and $\ph_n$ satisfy \eqref{eq:f}
respectively \eqref{est:ph}  with constants uniform in $n$, as well as
\[|\partial_t f_n(t,x,\zeta)|\le c\ell(t_0)(1+|\zeta|) \quad\text{ and }\quad
           |\partial_t \ph_n(t,x,\zeta)|\le c\ell(t_0)(1+|\zeta|^2)\]
 for  $n \in \mathbb N$ and a.e.\ $t\in[-t_0, t_0]$. Here and below in this proof we denote by $c$ and $C$ constants which may vary from line to line and  depend only on $t_0>0$ and quantities given by \eqref{eq:f} and \eqref{eq:sign}.
In particular, they do not depend on $n$, $t\in[0,t_0]$, or $w_0$.
Set $K_n(t,w):=(0,-f_n(t,\cdot,u))$.
Observe that  $K_n:[-t_0,t_0]\times X\to X$ is (globally) Lipschitz for each $n \in \mathbb N$ by the properties of $f_n$.

Let $w_0=(u_0,u_1)\in X$ and consider a sequence $(w_{0,n})_{n=1}^\infty$ in $D(A_+)$ with
$w_{0,n}=(u_{0,n},u_{1,n})$ such that  $w_n \to w_0$ in $X$ as $n \to \infty$.
Due to \cite[Theorem~6.1.6]{Pazy}, problem \eqref{eq:excited1} for $K_n$ and $w_{0,n}$  has a (unique) solution
$w_n\in C^1([0, \infty),X)\cap C([0,\infty),D(A_+))$, where $D(A_+)$ is endowed with the graph norm.
It is now easy to see that $w_n=(u_n,\partial_t u_n),$ where
\[u_n\in C^2([0, \infty), L^2(\cM)) \cap C^1([0, \infty),V)\cap C([0, \infty), H^2(\cM))\]
solves  \eqref{eq:excited} for $f_n$ and $w_{0,n}$.

We show that the solutions $w_n$ are bounded in $X$ uniformly in $n\in\NN$ and  then pass to the limit as $n\to\infty$.
To this aim, we look at the energies
\[
 E^0(w)=\frac12\int_\cM (|\nabla u|^2+|\partial_t u|^2)dx, \qquad
   E_n(t,w)= E^0(w)+\int_\cM\ph_n(t,x,u)dx 
	\]
for $w=(u,v)\in X$.   The above observations imply that for all $w \in X$ and $t>0$ one has $E^0(w)\le E_n(t,w)\le c(1+\|w\|_X^{1+\alpha})$,
$E_n(t, \cdot), E^0(\cdot)\in C^1(X,\RR)$, 
and $E_n(\cdot,w)\in W^{1,\infty}_{\mathrm{loc}}(\RR)$, where we also use that $H^1(\mathcal M)\hookrightarrow L^{\alpha+1}(\mathcal M)$
in view of $\alpha+1 < 2d/(d-2)_+$.
For   a.e.\ $t\in[0, t_0 ]$, we compute
\begin{align*}
\tfrac{d}{dt} E_n(t,w_n(t))
  &=\int_\cM\big[\partial_t u_n(t) \cdot [\partial_{tt} u_n(t)-\Delta u_n(t)+ f_n(t,\!\cdot,u_n(t))]\\
&\qquad    + (\partial_t\ph_n)(t,\cdot,u_n(t))\big] dx\\
  &= \!\int_\cM\big[b|\partial_tu_n(t)|^2 +  (\partial_t\ph_n)(t,\cdot,u_n(t))\big]dx\\
  &\le c(\|(u_n(t),\partial_tu_n(t))\|_{L^2}^2+1),\\
E_n(t,w_n(t))&\le E_n(0, w_{0,n}) +c+ c\int_0^t \big(\|u_n(s)\|_{L^2(\cM)}^2+E_n(s,w_n(s))\big) ds.
\end{align*}
We can bound the $L^2$-norm of $u_n(t)$ by means of
$u_n(s)=u_{0,n}+\int_0^s \partial_t u_n(\tau)\, d\tau$, so that
\[ 
E_n(t,w_n(t))\le cE_n(0, w_{0,n}) +c+ c\int_0^t E_n(s,w_n(s)\big) ds.
\]
Gronwall's inequality then yields
\begin{equation*}
E^0(t,w_n(t))\le  E_n(t,w_n(t))\le c(\|w_{0,n}\|_X^{1+\alpha}+1) e^{ct}\le  C(t_0)(\|w_{0}\|_X^{1+\alpha}+1),
\end{equation*}
and hence
\begin{equation}\label{eq:energy}
\|w_n(t)\|_X^2 = \|u_n(t)\|_{L^2}^2+ 2 E^0(t,w_n(t)) \le C(t_0)(\|w_{0}\|_X^{1+\alpha}+1)
\end{equation}
for $0\le t\le t_0$ and constants not depending on $n$ and $w_0$.

Using \eqref{eq:energy}, we can find a subsequence of $(w_n)_{n=1}^\infty$, denoted by the same symbol,
such that $u_n$ and $\partial_t u_n$ tend to functions $u$ and $v$ weakly* in each $L^\infty([0,t_0], V)$
respectively $L^{\infty}([0,t_0], L^2(\cM))$,
and $w=(u,v)$ satisfies \eqref{eq:energy} for a.e.\ $t\in [0,t_0]$. It is straightforward to check that
$v=\partial_t u$, and so $w_n$ converges to $w=(u,\partial_t u)$ weakly* in
$L^\infty([0,t_0], X)$ as $n \to \infty$.
Moreover, by the uniform bound \eqref{eq:energy} and the compact embedding from \cite[Corollary~4]{Simon}
the sequence $(u_n)$ is precompact in  $C([0,t], L^{2\alpha}(M))$, hence in $L^{2\alpha}([0,t_0] \times M)$.
There thus exists a subsequence, also denoted by  $(u_{n})$,
such that $u_{n}\to u$ in $L^{2\alpha}([0,t_0]\times \cM)$ and pointwise a.e.,
as well as $|u_{n}|\le g$ a.e.\ for a function $g\in L^{2\alpha}([0,t_0]\times \cM)$. 

To show convergence of $(f_{n}(t,\cdot,u_{n}(t)))$ and thus of $K_{n}(t, w_{n})$,
let $\e>0$ and take $t\in[0,t_0]$ such that
$(u_{n}(t,\cdot))$ tends to $u(t,\cdot)$ a.e.\ and is bounded a.e.\ by $g(t,\cdot)\in L^{2\alpha}(\cM)$.
Choose $x\in \cM$ outside this null set. For sufficiently large $n$ we have $|u(t,x)|\le n-1$. There is
a number $\delta\in(0,1)$ such that $|f(t,x,\zeta)-f(t,x,u(t,x))|\le \e $ if $|\zeta-u(t,x)|\le \delta$
since $f$ is continuous in $\zeta$. So there exists an index $n_\e$ such that for all $n\ge n_\e$
and $|\eta|\le 1/n$ we have $|u_{n}(t,x)-\eta|< n$ and
\begin{align*}
|\tilde{f}_{n}(t,x,u_{n}(t,x)-&\eta)- f(t,x,u(t,x))| \\
&= |f(t,x,u_{n}(t,x)-\eta)- f(t,x,u(t,x))| \le \e.
\end{align*}
 It follows that $|f_{n}(t,x,u_{n}(t,x))-f(t,x,u(t,x))| \le \e$, and hence
$f_{n}(t,\cdot,u_{n}(t))$ tends pointwise a.e.\ to  $f(t,\cdot,u(t))$ as $n\to\infty$, for a.e.\ $t\in [0,t_0]$.

Since $f_n$ satisfies \eqref{eq:f} uniformly in $n$, we have
$|f_{n}(t,\cdot,u_{n}(t))|\le c\kappa(t)(1+g(t,\cdot)^\alpha)$
and thus $f_{n}(t,\cdot,u_{n}(t))\to f(t,\cdot, u(t))$ in $L^2(\cM)$ for a.e.\ $t\in [0,t_0]$. 
Moreover, by \eqref{eq:f} and \eqref{eq:energy} the map $t\mapsto \|f_n(t,\cdot,u_n(t))\|_{L^2(\cM)}$ is locally bounded independent of $n$. Therefore the right-hand side  in the formula for mild solutions
\[w_{n}(t)=T_+(t)w_{0,{n}}+ \int_0^t T_+(t-s) K_{n}(s, w_{n}(s))ds\]
converges in $C([0,t_0],X)$  as $n\to\infty$. As seen above, the left-hand side tends to
 $w=(u,\partial_t u)$ weakly* in $L^\infty([0,t_0],X)$  so that
 \[w(t)=T_+(t)w_{0}+ \int_0^t T_+(t-s) K(s, w(s))ds\]
 holds for a.e.\ $t\in[0,t_0]$. Since the right-hand side is continuous and $t_0>0$ is arbitrary,
 $w$ solves \eqref{eq:excited1} and satisfies \eqref{eq:energy} for all $0\le t\le t_0$.
\end{proof}

Combined with the above analysis and known spectral properties, Theorem \ref{main} now shows growth of orbits
and thus global instability in concrete examples. One could also use Theorem \ref{subseq}, adding more assumptions
on $f$ to satisfy (A1). We avoid doing so since the modifications are easy.
Note that in view of \eqref{sr_adjoint_wave},
we can switch freely between the two equations \eqref{eq:excited1} and
\eqref{eq:backward1}. So we concentrate just on \eqref{eq:backward1}.

As the first (toy) example, we  improve Theorem~III in \cite{Sola}
in various respects: We can treat time-depending $f$, reduce the regularity requirements in $\zeta$ from $C^2$ to
continuity, and  obtain exponential growth instead of mere unboundedness.
In all examples  we assume that $f$ satisfies \eqref{eq:f} and either \eqref{eq:lin} or \eqref{eq:sign}.

\begin{example}\label{sola}
 Let  $b>0$ be constant and $\mathcal M$ be 
a $d$-dimensional, compact, smooth and connected Riemannian manifold without or with boundary.
 It is straightforward to check that the (point) spectrum of $A_-$
 is given by an unbounded sequence of eigenvalues on $-\frac{b}2+i\RR$ and at most finitely many points in $(-b,0]$.
 It follows $s_R(-A_-)\ge \frac{b}2$. By Corollary~\ref{main1} and Propositions~\ref{lem:K1} and \ref{lem:K2},
given a function $a:[0,\infty)\to [0,\infty)$ decreasing to $0,$
 there are a dense set of initial values $w_0\in X$ and $t_0 \ge 0$ such that the corresponding solutions $w(t,w_0)$ of
 backward damped wave equation  \eqref{eq:backward1} admit the lower bound
$\|w(t,w_0)\|\ge a(t)e^{bt/2}$ for  $t\ge t_0.$ In other words, the energy of solutions to \eqref{eq:backward} grows
exponentially in backward time for a dense of initial values.
\end{example}

The next example generalizes the preceding one by allowing the damping $b$ to be non-stationary 
and far from being smooth. At the same time, it concerns only the  case $d=1.$

\begin{example}\label{cox}
Let now $\mathcal M=[0,1]$ and $b \in {\rm BV}([0,1])$ with $b > 0.$
(Recall that we then impose Dirichlet boundary conditions.)
It was proved in \cite[Theorem 5.3]{Cox} that there is a sequence of eigenvalues
$(\lambda_n)_{n=1}^\infty \subset \sigma(A_-)$ with
${\rm Re}\, \lambda_n \to -\frac\beta2$ as $n \to \infty$ where  $\beta:=\int_{0}^{1}b(s) ds.$
So either $-\frac \beta2+i\mathbb R$ contains an infinite number of the eigenvalues of $A_-,$
or the resolvent of $A_-$ is unbounded on $-\frac \beta2+i\mathbb R,$
and therefore, $s_R(-A_-)\ge \frac\beta2.$ As above, using  Corollary~\ref{main1} along with Propositions~\ref{lem:K1}
and \ref{lem:K2}, for a given  function $a:[0,\infty)\to [0,\infty)$ decreasing to $0$ we find
a dense set of initial values $w_0$ such the solutions of \eqref{eq:backward1} satisfy
$\|w(t,w_0)\|\ge a(t)e^{\beta t/2}$ for $t \ge t_0$ and some $t_0\ge0$.
Moreover, by \cite[Theorem 3.4]{Freitas}, the same result holds without assuming $b > 0$ if $\|b\|_\infty$ is
 sufficiently small. 
\end{example}

We proceed with more involved frameworks,  relying on quite subtle results
from the spectral theory of damped wave equations.

\begin{example}\label{sjostrand}
Let $\mathcal M$ be a manifold as in Example \ref{sola}, having no boundary.
 (One may also consider manifolds with boundary and the corresponding \emph{generalised} geodesic flows,
but this setting leads to technical complications, and is thus omitted for simplicity.)
To explain our next example, we need to introduce several auxiliary notions pertaining
to dynamics of the (Hamiltonian) geodesic flow $(g^t)_{t \in \mathbb R}$ on a Riemannian cosphere bundle $\mathcal S^* \mathcal M$
over $\mathcal M.$ A relevant discussion of geodesic flows can be found e.g. in \cite[Appendix B]{Laurent},
see also \cite[Section 2.1]{Bessa}.
Write
\[ \rho_t=(x_t,\xi_t)=g^t(\rho_0), \qquad  \rho_0=(x_0,\xi_0) \in S^*\mathcal M,\, t \in \mathbb R,\]
and let $\pi: S^* \mathcal M \to \mathcal M$ be a canonical projection.
Given a damping $b \in C^\infty (\mathcal M)$ with $b \ge 0,$ define its Birkhoff ergodic average over the geodesic curve in
$\mathcal M$ as $\langle b \rangle_t (\rho_0): = \frac{1}{2t}  \int^t_{-t} (b\circ \pi \circ g^s)(\rho_0)\, ds, t >0,$
and let
\begin{align*}
b_{-}&:= \sup_{t>0} \inf_{\rho_0 \in S^* \mathcal M} \langle b \rangle_t (\rho_0)
      =\lim_{t \to \infty}\inf_{\rho_0 \in S^* \mathcal M}\langle b \rangle_t (\rho_0),\\
b_{+}&:= \inf_{t>0}\sup_{\rho_0 \in S^* \mathcal M} \langle b \rangle_t (\rho_0)
       =\lim_{t \to \infty}\sup_{\rho_0 \in S^* \mathcal M}\langle b \rangle_t (\rho_0).
\end{align*}
Note that by the Birkhoff ergodic theorem, $\langle b \rangle_\infty (\rho_0):=\lim_{t \to \infty} \langle b \rangle_t (\rho_0)$
exists almost everywhere with respect to the flow invariant (normalised) Liouville
measure on $S^*\mathcal M$, and setting 
\[ 
b^-_{\infty}:={\rm ess \, inf}_{\rho_0 \in S^* \mathcal M}\, \langle b \rangle_\infty, \qquad
b^+_{\infty}:= {\rm ess \, sup}_{\rho_0 \in S^* \mathcal M} \, \langle b \rangle_\infty,
\]
we have
\[b_- \le  b^-_{\infty} \le b^+_{\infty}\le b_+,\]
where all of the inequalities can in general be strict. If the geodesic flow is ergodic, then one has
\[
b^-_{\infty} =b^+_{\infty}
  =\frac{1}{{\rm vol}\, (\mathcal M)} \int_{\mathcal M} b(x)\, dx:=b^*_\infty.
	\]
It was proved in \cite{Lebeau} that for every $\epsilon >0$ there are at most finite number of the eigenvalues
of $A_-$ outside the strip $[-\frac{b_+}{2} -\epsilon,-\frac{b_-}{2}+\epsilon]+i\mathbb R$ (where, in particular $\mathcal M$
may have a boundary). So that there is a sequence of the eigenvalues clustering at $\beta+i\mathbb R$ for some
$\beta \in [-\frac{b_{+}}{2}, -\frac{b_-}{2}]+i\mathbb R,$
and then arguing as in Example \ref{cox} one concludes that $s_R(-A_-)\ge \frac{b_-}{2}.$
We also refer to \cite{Sjostrand} for comments on the generality of this result and an alternative proof.

The result was improved in \cite{Sjostrand} by showing that for every $\epsilon >0$ an infinite number of
 the eigenvalues of $A_-$ belong to the strip
$[-\frac{b^+_{\infty}}{2}-\epsilon, -\frac{b^{-}_{\infty}}{2}+\epsilon]+i\mathbb R.$
Moreover, as proved in \cite{Sjostrand}, if $(g^t)_{t \in \mathbb R}$ is ergodic, then the eigenvalues of $A_-$
cluster at $-\frac{b^*_\infty}{2}+i\mathbb R.$
(See also \cite{Asch} for an illuminating discussion of these results.)
Thus, observing that $s_R(-A_-) \ge \frac{b^-_{\infty}}{2},$ or, if $(g^t)_{t \in \mathbb R}$ is ergodic,
$s_R(-A_-)\ge \frac{b^*_\infty}{2}$, we get an exponential lower norm bound
for $w(t,w_0)$ given by \eqref{eq:backward1}. Clearly,  the two estimates  for $s_R(-A_-)$ considered in this example
may produce $s_R(-A_-)>0,$ and thus  yield a dense set of initial values
 for exponentially growing solutions to  \eqref{eq:backward1}.
\end{example}

So far, our examples depended on the properties of the spectrum of $A_-.$
However, there are interesting situations when one has to invoke the resolvent
of $A_-$ and thus to use a full strength of our Corollary~\ref{main1}. 

\begin{example}\label{dyatlov}
There are many examples in the literature where $(\dot T_{-}(t))_{t\ge 0}$
satisfies $\|\dot T_{-}(t)R(\mu_0, \dot A)^s\|\le M_s e^{-\omega_s t}$ for  $t \ge 0$
and some $s,\omega_s, M_s>0$ as well as, at the same time,
$\omega_0(\dot T_{-})\ge 0.$ In other words,
$(\dot T_{-}(t))_{t\ge 0}$ decays exponentially in the operator norm as a map from
$D((-\dot A_-)^s)$  for some $s>0$ to $X$,  but not in $B(X)$.
By interpolation, it then follows that $(\dot T_{-}(t))_{t \ge 0}$
enjoys such a decay for \emph{all} $s>0,$ and for us it suffices to fix $s=1.$

In view of e.g.\ \cite[Theorem 1.4]{Weiss}, for some $a,c>0$  the resolvent $R(\lambda, \dot A_{-})$
then extends analytically to $\{\lambda: \Re \lambda \ge -a\}$
and is norm bounded there by $c(1+|\lambda|).$
Note that $R(\lambda, \dot A_-)$ is  unbounded on $i\RR$. Indeed, if this was wrong, using
the Neumann's series expansion for $R(is, \dot A_{-})$, $s \in \mathbb R,$ and the resolvent estimate
$\|R(\lambda, \dot A_{-})\|\le ({\rm Re}\, \lambda)^{-1}$ for ${\rm Re}\, \lambda >0,$
we would infer that $s_0(\dot A_{-})<0,$ and then $\omega_0(\dot T_{-})<0$ by \eqref{type}.
By a standard application of Phragmen--Lindel\"of's theorem, $R(\lambda, \dot A_{-})$  also has
to be unbounded on a vertical line $-\beta+i\RR$ for some $\beta\in (0,a]$. As a result, the resolvent
of $-A_-$ is unbounded on  $\beta+i\RR,$ so that $s_R(-A_-)\ge \beta>0$. Hence, as above
for a dense set of initial values we infer that the norms of solutions to the backward damped wave
equation \eqref{eq:backward1} grow exponentially and thus show a global instability result.

To describe concrete situations when such an effect can happen recall that 
if $\mathcal M$ is as in Example \ref{sjostrand} and is negatively curved, then by a classical result due to Anosov (see e.g. \cite{Anosov}) a geodesic flow on $\mathcal {S^*\mathcal M}$ 
admits  countably many periodic orbits. It was revealed in \cite[Theorem 1]{Schenck}
that each such an orbit gives rise to a smooth damping $b$ such that $(\dot T_{-}(t))_{t \ge 0}$ does not decay
exponentially, while its orbits $\dot T_{-}(t)x$ decay exponentially for sufficiently  smooth initial data $x.$
More precisely, it was shown in \cite{Schenck} that, for any periodic geodesic   $\gamma$ in $\mathcal M$ and
$b_0\in C^{\infty}(\mathcal M),$ there exists $\epsilon >0$ such that if $b_0$ vanishes in an
$\epsilon$-neighborhood of $\gamma$ and is positive everywhere else, then such a decay takes place for
$b=cb_0$ for all sufficiently large $c>0.$ 

If $\mathcal M$ is a hyperbolic surface with constant negative curvature, 
then as proved in  \cite[Theorem 1.1]{Jin}, the decay takes place  for \emph{all} smooth dampings $b$.
This property can in fact be generalized to all surfaces $\mathcal M$ whose geodesic flow has the
so-called Anosov property, see \cite[Theorem 6]{Dyatlov}, though such a generalisation is very deep and demanding.
Other instances of the exponential decay for only smooth enough orbits of $(\dot T_{-}(t))_{t \ge 0},$ sometimes
with explicit rates, can also be found in \cite[Section 4]{BurqC}, \cite{Nonnen} and
\cite[Theorem 3 and the subsequent Remark]{SchenckP}. We avoid their discussion to keep our exposition within reasonable limits.
\end{example}

Note that the assumption $b \ge 0$ was chosen just to fit in the framework of the existing work,
and it can be avoided in many cases (e.g.\ in Examples \ref{cox}, \ref{sola}, and \ref{sjostrand}).
In this case, $s_R$ provides just a lower bound, not necessarily growing exponentially.

Finally, we  show non-stabilizability of certain nonlinear infinite-dimen\-sional control systems. In this way, we
generalize the corresponding results in  \cite{Guo} or \cite{Triggiani},
for instance, where the  operators $B,$  $F,$ and $C$ used below are linear and bounded. The literature on
stabilization of control systems is enormous (and, thus, we skipped a discussion of asymptotics for
damped wave equations as a stabilization problem, see e.g.\ \cite{LeR} and \cite{LeR1} concerning that).
We just refer to \cite{Lasiecka} and \cite{Slemrod}  as sample works on nonlinear stabilization and to
\cite{Coron} for general concepts
of nonlinear control. For basics of linear theory one may consult \cite[Section  VI.8]{EnNa00}.

\begin{example}\label{control}
Let $(T(t))_{t \ge 0}$ be a $C_0$-semigroup on a Banach space $X$ with generator $A,$ and let the
control map $B:U\to X$, the  feedback $F:Y\to U$ and the observation map $C:X\to Y$  be all (possibly) nonlinear and continuous
with linear growth (i.e., satisfy \eqref{eq:lingrowth} with a constant $c$), where $U$ and $Y$ are Banach spaces.
Combined with Corollary~\ref{thm:cp2} and Proposition~\ref{prop:growth}, Theorems \ref{theorem1} and \ref{main}
imply the following results:

If $s_R(A)=0,$ then the system $x'=Ax+ BFCx$ is not exponentially stabilizable by a compact
nonlinear feedback $F,$ i.e., it will always have solutions not decaying exponentially.
If $s_R(A)>0,$ the system is not strongly stabilizable by 
a compact nonlinear feedback $F,$ i.e., some of its solutions will not converge to zero.
(In fact, they will grow exponentially.)

The analogous statement holds for time-discrete feedback systems $x_{n+1}=Ax_n+ BFCx_n, n \ge 0,$ for bounded $A$
and continuous $B$, $F$, and $C$ mapping bounded sets into bounded sets, if one replaces
$s_R(A)$ with $r_e(A).$
\end{example}

\section{Appendix}\label{sec:app}

In this section we prove several results on geometric properties of Banach spaces and 
fine spectral theory of semigroups and their generators which are crucial for our lower estimates.
Some of them, e.g.\  Theorem \ref{proplower}, are of independent interest.
The exposition here follows \cite[Section 4]{Muller_Tom} with appropriate changes and improvements,
which warrant an independent treatment. The section makes the paper essentially self-contained.

First, we note a geometric statement from Banach space theory.
Its versions are often used in iterative constructions
arising in the study of orbits of linear operators,
and it are important in our studies too.
To make our presentation self-contained and to provide a better understanding 
of our constructions, we give its proof below.
Let $X$ be a Banach space. Recall that a subset $\Lambda\subset X^*$ is norming if
$$
\|x\|=\sup\Bigl\{\frac{|\langle x,y\rangle|}{\|y\|}: y\in \Lambda, y\ne 0\Bigr\}
$$
for all $x\in X$. Part (a) of the next result is \cite[Lemma V.37.6]{Muller}.
We add two related statements whose proof is a variation of the arguments in \cite{Muller}.
\begin{lemma}\label{geomlemma}
 Let $F$ be a finite-dimensional subspace of a Banach space $X$ and $\e>0$.
\begin{itemize}
\item [(a)] There exists a closed subspace $M\subset X$ of finite
codimension such that
\begin{equation}\label{eq:star}
\|f+m\|\ge(1-\e)\max\bigl\{\|f\|,{\|m\|/2}\bigr\}
\end{equation}
for all $f\in F$ and $m\in M$.
\item [(b)] Let $M$  be given by (a) and
$E$ be a finite-dimensional subspace with $F \subset E$. Then there exists $L\subset M$ of  finite
codimension such that \eqref{eq:star} holds for $f \in E$ and $m \in L$.
\item [(c)] If $\Lambda\subset X^*$ is a norming set, then there exists
a subspace $M\subset X$ satisfying \eqref{eq:star} such that
$M=\bigcap_{j=1}^k{\rm Ker} \, y^*_j$ for some $y^*_1,\dots,y^*_k\in\Lambda$.
\end{itemize}
\end{lemma}

\begin{proof}
 We can assume that $\e<1$. The unit sphere in $F$ is compact, therefore there
 exists a finite subset $D\subset \{f\in F:\|f\|=1\}$ satisfying
 $\dist\{e,D\}\le\e/2$ for all $f\in F$ with  $\|f\|=1$. Let $\Lambda\subset X^*$ be norming.
 For each $d\in D$ there is a functional $y^*_d\in \Lambda$ such
that $\bigl|\bigl\langle d,y^*_d/\|y^*_d\|\bigr\rangle\bigr|>1-\e/2$.
 Set $M=\bigcap_{d\in D}{\rm Ker}\, y^*_d$. Clearly, $M$ is closed with finite codimension.

  To prove the required inequality,
let $f\in F$ and $m\in M$. We can assume that $\|f\|\ne 0$ since
the assertion is clear for $f=0$.
 Choose $d\in D$ with $\bigl\|d-f/\|f\|\bigr\|\le\frac\e2$.
As in \cite[Lemma V.37.6]{Muller}, we then estimate
 \begin{align*}
 \|f+m\|&\ge \Bigl|\Bigl\langle f+m,\frac{y^*_d}{\|y^*_d\|}\Bigr\rangle\Bigr|
 =\Bigl|\Bigl\langle f-\|f\|d,\frac{y^*_d}{\|y^*_d\|}\Bigr\rangle
   +\Bigl\langle\|f\|d,\frac{y^*_d}{\|y^*_d\|}\Bigr\rangle\Bigr|\\
         &\ge \|f\|(1-\e/2)-\bigl\|f-\|f\|d\bigr\|\ge \|f\|(1-\e).
 \end{align*}
One can now conlcude as in  \cite[Lemma V.37.6]{Muller} to show (a) and (c).
Part (b) is is direct consequence of the construction above.
\end{proof}

If $X$ is a Hilbert space,  we can take $M=F^\perp$ in the  lemma.
Thus the subspace $M\subset X$ constructed above plays a role of the
orthogonal complement of a finite-dimensional subspace for
general Banach spaces.

Next we turn to the spectral theory of semigroups related to constrution of 
approximate eigenvectors with some additional geometric properties.
First we prove a property of the boundary essential spectrum of unbounded 
operators, well-known in the bounded case. 
\begin{lemma}\label{codim}
Let $A$ be a closed, densely defined operator on a Banach space $X,$
such that
$\rho(A)\ne\emptyset$, and $\lambda\in\partial\sigma_e(A)$. Let $M\subset X$
be a closed subspace of finite codimension and $\epsilon>0$. Then there exists a unit
vector $x\in M\cap D(A)$ such that $\|(A-\lambda)x\|<\epsilon$.
\end{lemma}
\begin{proof}

We will rely on the fact, that the statement is true if $A$ is bounded,
see e.g.\ \cite[Proposition III.19.1 and Theorem III.16.8]{Muller}.
Without loss of generality we may assume that $\la=0$.

Let $\mu\in\rho(A)$ and let $T:=A(A-\mu)^{-1}=I+\mu(A-\mu)^{-1}$. Then $T$
is bounded. Moreover,
$\sigma_e(T)\setminus\{1\}=\{1+\frac{\mu}{z-\mu}:z\in\sigma_e(A)\}$
by \eqref{smp_resolv}. So $0\in\partial\sigma_e(T)$.
Let $M'={\rm Im}\, ((A-\mu)\upharpoonright_M)$. Then ${\rm codim}\, M'<\infty$, and as  $T \in B(X)$
there exists a sequence  $(x_n)_{n=1}^\infty\subset M'$ such that  $\|x_n\|=1$ for all $n$ and $Tx_n\to 0$
as $n \to \infty$. Set $y_n=(A-\mu)^{-1}x_n$ for  $n \in \NN$. Then $y_n\in M\cap D(A)$ for all $n$ and
$Ay_n\to 0$ as $n \to \infty.$ Moreover, since
$$Tx_n=x_n+\mu(A-\mu)^{-1}x_n = x_n+\mu y_n,$$ 
we have $\liminf_{n \to \infty}\|y_n\|=1/|\mu|>0.$ 
It remains to choose $a y_n$ for appropriate $a>0$ and $n.$
\end{proof}

The proof of Theorem \ref{proplower} is based on the next lemma.
\begin{lemma}\label{apprspec}
Let $A$ generate the $C_0$-semigroup $(T(t))_{t\ge 0}$ on a Banach
space $X$. Let either $\omega >s_e(A)$ or $\omega=s_e(A)$ and 
$(\omega+i\mathbb R)\cap \sigma_{e}(A)=\emptyset$. Assume that  either
\[
\sigma_p(A)\cap (\omega+i\mathbb R) \,\, \text{is infinite} 
\]
or 
\[
 \sigma_p(A)\cap (\omega +i\mathbb R)  \text{ \ is at most finite \quad and} 
    \quad \limsup_{|b| \to \infty} \|R(\omega +ib, A)\|=\infty. 
\]
Then there exist $(\mu_n)_{n=1}^\infty\subset\omega+i\mathbb R$ and
$(u_n)_{n=1}^\infty\subset D(A)$ such that
  $\|u_n\|=1$ for all $n \in \mathbb N,$ $\|(\mu_n-A)u_n\|\to 0$ as $n \to \infty$,   
  and for every $y^* \in D(A^*)$ we have
   $$\langle u_n, y^* \rangle \to 0, \quad n \to \infty.$$
 In particular, if $X$ is reflexive, then $(u_n)_{n=1}^\infty$ tends  weakly to 0 (since then
 $D(A^*)$ is dense in $X^*$). 
\end{lemma}
\begin{proof}
By our assumptions, there exist a sequence $(b_n)_{n=1}^\infty$ with $|b_n|\to\infty$ and unit vectors
$u_n\in D(A)$ for $n \ge 1$ such that $\|(\omega +ib_n -A)u_n\|\to 0$ as $n \to \infty$.
Set $\mu_n=\omega+ib_n$. We show that
$$\langle u_n,y^*\rangle\to 0, \quad n \to \infty,$$
for each $y^*\in D(A^*)$. Let $y^*\in D(A^*)\subset X^*$ have norm 1.
Pick a vector $y\in D(A)$ with $\langle y,y^*\rangle> \frac12$. 
Let $M={\rm Ker}\, y^*$. Write $u_n=m_n+\al_n y$ for some $m_n\in M$
and  $\al_n\in\mathbb C$. Then
the sequences $(m_n)_{n=1}^\infty$ and $ (\al_n)_{n=1}^\infty$ are
bounded. Furthermore,
$$
\langle(\mu_n-A)u_n,y^*\rangle\to 0, \qquad n \to \infty,
$$ 
and
\begin{align*}
 \langle(\mu_n-A)u_n,y^*\rangle&=\langle(\mu_n-A)m_n,y^*\rangle+\al_n\langle(\mu_n-A)y,y^*\rangle\\
&= \al_n\mu_n\langle y,y^*\rangle-\langle m_n,A^*y^*\rangle-\al_n\langle Ay,y^*\rangle.
\end{align*}
Since the last two terms are uniformly bounded  and
$|\mu_n|\to\infty$, we have $\al_n\to 0$ as $n \to\infty$. 
It follows $\langle u_n,y^*\rangle=\alpha_n \langle y,y^*\rangle\to 0$.
\end{proof}

Let  $(T(t))_{t\ge 0}$ be a $C_0$-semigroup on a Banach space $X$, and set
\begin{equation*}
\|x\|_1:=\sup \{|\langle x, x^* \rangle|: x^* \in D(A^*), \|x^*\|\le 1\}
\end{equation*}
for $x\in X$. 
Then $\|\cdot\|_1$ is an equivalent norm satisfying
\begin{equation}\label{renorm}
\|x\|_1\le\|x\|\le \al\|x\|_1, \qquad \text{where \ }  \alpha:=\limsup_{t \to 0}\|T(t)\|,
\end{equation}
for all $x\in X$. See e.g.\ \cite[Theorems 1.3.1 and 1.3.5]{Ne92}.
Hence, renorming $X$ with $\|\cdot \|_1,$ we can make $D(A^*)$ a norming set for the Banach space $(X,\|\cdot\|_1)$.

Recall the definitions of the resolvent bound $s_R$ and of the notion of admissibility
given in Section \ref{prelim}.
Now we are ready to prove Theorem \ref{proplower} stated there and 
describing one of admissible $\omega$
in resolvent terms. More precisely, we show that if
$A$ generates a $C_0$-semigroup $(T(t))_{t\ge 0}$ on a
Banach space $X,$ then the number $s_R(A)$ given by \eqref{defsr}  is admissible.
\smallskip

\emph{Proof of Theorem \ref{proplower}.} \quad
Let $M$ be a closed subspace of finite codimension in $X$ given by
$M=\bigcap_{j=1}^k{\rm Ker}\, y^*_j$ for some functionals $y^*_1,\dots,y^*_k \in D(A^*).$
Let $\e >0$ and  $t_0 >0$ be fixed,
and set   $K:=\sup\{\|T(t)\|:0\le t\le t_0\}$.
Since the admissibility does not depend on equivalent renormings of the underlying space,
in view of \eqref{renorm}, we can assume that $D(A^*)$ is a norming set for $X.$

Let $s_R(A)=s_e(A)$ and  $\mu \in \sigma_{e}(A)\cap (s_R(A)+i\mathbb R)$.
Lemma \ref{codim} then yields  $x \in D(A)\cap M$ with $\|x\|=1$ and
$\|(A-\mu)x\|\le \e.$ Hence 
$$
\|T(t)x-e^{\mu t}x\|= \Bigl\|\int_0^t e^{\mu(t-s)}T(s)(\mu-A)xds\Bigr\|
 \le \e t_0 e^{s_R(A)t_0}K
$$
for all $t\in [0, t_0]$.

Let $s_R(A)\ge s_e(A)$ and $\sigma_{e}(A)\cap (s_R(A)+i\mathbb R) =\emptyset$.
Employing 
Lemma \ref{apprspec}, we find sequences $(\mu_n)_{n=1}^\infty\subset\mathbb C$
with $\mu_n=s_R(A)+i b_n$ for $n \in \mathbb N$ and $(u_n)_{n=1}^\infty\subset D(A)$ with
$\|u_n\|=1$ such that
\[
 \langle u_n, y^* \rangle \to 0 \quad  \text{for every} \,\,  y^* \in D(A^*) \quad \text{and} \qquad
\|(\mu_n-A)u_n\|\to 0, \quad n \to \infty.
\]

By \cite[Lemma 7.4]{Schechter},
there exists a finite-dimensional subspace $F\subset D(A)$ such that
$X=M\oplus F$.
Let $P$ be the projection onto $F$ with ${\rm Ker}\,P= M$.
 By the choice of $u_n$ we have
$\|Pu_n\|\to 0$ so that $\|(I-P)u_n\|\to 1$ as $n \to \infty$ and
$$\Bigl\|u_n-\frac{u_n-Pu_n}{\|u_n-Pu_n\|}\Bigr\|\to 0, \qquad n \to \infty.$$ 
Choose $n_0\in\NN$ such that
\begin{align*}
\Bigl\|u_{n_0}-\frac{u_{n_0}-Pu_{n_0}}{\|u_{n_0}-Pu_{n_0}\|}\Bigr\|
   &\le\min\left\{\frac{\e}{4K}, \frac{\e}{4e^{s_R(A)t_0}}\right\},\\
\|(\mu_{n_0}-A)u_{n_0}\| &\le \frac{\e}{4 t_0K\max\{1,e^{t_0s_R(A)}\}}.
\end{align*}
Set 
$$\mu=\mu_{n_0}\quad  \text{and}\quad 
    x=\frac{u_{n_0}-Pu_{n_0}}{\|u_{n_0}-Pu_{n_0}\|}.$$
For every $0\le t\le t_0,$ we have
$$
\|T(t)u_{n_0}-e^{\mu  t}u_{n_0}\|
= \Bigl\|\int_0^t e^{\mu(t-s)}T(s)(\mu_{n_0}-A)u_{n_0} ds\Bigr\|
 \le \e/4
$$
and
\begin{align*}
\|T(t)x-e^{\mu t}x\|&\le \|T(t)x-\!T(t)u_{n_0}\|+ \|T(t)u_{n_0}\!-e^{\mu t}u_{n_0}\|
+ \|e^{\mu t}u_{n_0}\!-e^{\mu t}x\| \\
 &\le K\|x-u_{n_0}\|+\e/4
   +e^{ts_R(A)}\|x-u_{n_0}\|\\
&<\e.
\end{align*}
This finishes the proof. $\hfill$ $\qed$

\begin{remark}
If $X$ is a reflexive Banach space, then the proof of the previous theorem is simpler. 
In this case it is not necessary to do the renormalization since $\|\cdot\|_1=\|\cdot\|$
by the density of $D(A^*)$, and the subspace $M$  in the statement of Theorem \ref{proplower}
can be any close subspace of finite codimension.
\end{remark}
\section{Acknowledgements}
We are grateful to the referee for valuable suggestions and remarks which
led to improvements of this paper. 
\section{Conflict of interests statement}

On behalf of all authors, the corresponding author states that there is no conflict of interest.

\section{Data availability statement}

Data sharing not applicable to this article as no datasets were generated or analysed during the current study.

\section{Competing interests statement}

The authors have no competing interests to declare that are relevant to the content of this article.

\end{document}